\documentclass[11pt]{article}
\usepackage{amsmath, amssymb, amsthm}
\usepackage{enumerate}

\usepackage{tikz}
\usetikzlibrary{arrows}

\date{}
\author{
Michael Krivelevich \thanks{ School of Mathematical Sciences,
Raymond and Beverly Sackler Faculty of Exact Sciences, Tel Aviv
University, Tel Aviv 69978, Israel. Email address:
krivelev@post.tau.ac.il. Research supported in part by USA-Israel
BSF Grant 2010115 and by grant 1063/08 from the Israel Science
Foundation. } \and Choongbum Lee \thanks{Department of Mathematics,
UCLA, Los Angeles, CA, 90095. Email: choongbum.lee@gmail.com.
Research supported in part by a Samsung Scholarship.} \and Benny
Sudakov
\thanks{Department of Mathematics, UCLA, Los Angeles, CA 90095.
Email: bsudakov@math.ucla.edu. Research supported in part by NSF
grant DMS-1101185, NSF CAREER award DMS-0812005 and by USA-Israeli
BSF grant. } }

\numberwithin{equation}{section}
\numberwithin{figure}{section}
\theoremstyle{plain}
\newtheorem{thm}{Theorem}[section]
  \theoremstyle{plain}
  \newtheorem{lem}[thm]{Lemma}
  \theoremstyle{definition}
  \newtheorem{defn}[thm]{Definition}
  \theoremstyle{plain}
  \newtheorem{prop}[thm]{Proposition}
  \theoremstyle{plain}
  \newtheorem{claim}[thm]{Claim}
  \theoremstyle{plain}

\newcommand{\BIPRE}{\mathcal{RE}_b}
\newcommand{\BFP}{\mathbf{P}}

\begin{document}

\title{Robust Hamiltonicity of Dirac graphs}

\maketitle

\begin{abstract}
A graph is Hamiltonian if it contains a cycle which passes through
every vertex of the graph exactly once. A classical theorem of Dirac
from 1952 asserts that every graph on $n$ vertices with minimum degree at
least $n/2$ is Hamiltonian. We refer to such graphs as Dirac graphs.
In this paper we extend Dirac's theorem in two directions and show
that Dirac graphs are robustly Hamiltonian in a very strong sense.
First, we consider a random subgraph of a Dirac graph obtained by taking each
edge independently with probability $p$, and prove that there exists
a constant $C$ such that if $p \ge C \log n / n$, then a.a.s.~the
resulting random subgraph is still Hamiltonian. Second, we prove
that if a $(1:b)$ Maker-Breaker game is played on a Dirac graph,
then Maker can construct a Hamiltonian subgraph as long as the bias $b$ is at most $cn /\log n$
for some absolute constant $c > 0$. Both of these results are
tight up to a constant factor, and are proved under one general
framework.
\end{abstract}

\section{Introduction} \label{section_intro}

A \emph{Hamilton cycle} of a graph is a cycle which passes through
every vertex of the graph exactly once, and a graph is
\emph{Hamiltonian} if it contains a Hamilton cycle. Hamiltonicity is
one of the most central notions in Graph Theory, and has been intensively
studied by numerous researchers. The problem of deciding
Hamiltonicity of a graph is one of the NP-complete problems that
Karp listed in his seminal paper \cite{Karp}, and accordingly, one
cannot hope for a simple classification of such graphs. Still, there
are many results deriving properties sufficient for Hamiltonicity.
For example, a classical result proved by
Dirac in 1952 (see, e.g., \cite[Theorem~10.1.1]{Diestel}), asserts that
every graph on $n$ vertices of minimum degree at least $\frac{n}{2}$
is Hamiltonian. In this context, we say that a graph is a {\em Dirac
graph} if it has minimum degree at least $\frac{n}{2}$. Dirac's
theorem is one of the most influential results in the study of
Hamiltonicity and by now there are many related known results (see, e.g.,
\cite{Bondy2}).

Let $G$ be a graph and $\mathcal{P}$ be a graph property. Many
results in Graph Theory state that ``under certain conditions, $G$
has property $\mathcal{P}$''. Once such a result is established, it
is natural to ask: ``How strongly does $G$ possess $\mathcal{P}$?''.
In other words, we want to determine the robustness of $G$ with
respect to $\mathcal{P}$. Recently, such questions were extensively
studied by many researchers.

In order to answer the question about robustness, we would need some
kind of a measure of this phenomenon. There are several measures of
robustness that were proposed so far. For example, one can measure
the robustness of Dirac graphs with respect to Hamiltonicity by
computing the number of Hamilton cycles that a Dirac graph must
contain. Indeed, confirming a conjecture of S\'ark\"ozy, Selkow, and
Szemer\'edi \cite{SaSeSz}, Cuckler and Kahn \cite{CuKa} proved that
every Dirac graph contains at least $n!/(2+o(1))^n$ Hamilton cycles.
Another measure is the so called {\em resilience}, whose systematic
study was initiated by Sudakov and Vu \cite{SuVu}, and has been
intensively studied afterwards, see, e.g., \cite{AlSu, BaCsSa,
BeKrSu, DeKoMaSt, KrLeSu} and their references 
(resilience is closely related to the notion of fault tolerance,
see, e.g., \cite{Aletal}). Roughly speaking,
for monotone increasing graph properties, these measures compute the
robustness in terms of the number of edges one must delete from $G$
locally or globally in order to destroy the property $\mathcal{P}$.
In this paper, we would like to revisit Dirac's theorem and study
different settings which can be used to demonstrate its robustness.
Our main results show how to strengthen Dirac's theorem in two ways.

\subsection{Random subgraph}

Let $G(n,p)$ be the binomial model of random graphs, which denotes
the probability space whose points are graphs with vertex set $[n] =
\{1,\ldots,n\}$ where each pair of vertices forms an edge randomly
and independently with probability $p$. We say that $G(n,p)$
possesses a graph property $\mathcal{P}$ \textit{asymptotically
almost surely}, or a.a.s.~for brevity, if the probability that
$G(n,p)$ possesses $\mathcal{P}$ tends to 1 as $n$ goes to infinity.
The earlier results on Hamiltonicity of random graphs were proved by
P\'osa \cite{Posa}, and Korshunov \cite{Korshunov}. Improving on
these results, Bollob\'as \cite{Bollobas}, and Koml\'os and
Szemer\'edi \cite{KoSz} independently proved that if $p\ge\frac{\log
n + \log \log n + \omega(n)}{n}$ for some function $\omega(n)$ that
goes to infinity as $n$ goes to infinity, then $G(n,p)$ is
a.a.s.~Hamiltonian. The range of $p$ cannot be improved, since if $p
\le \frac{\log n + \log \log n - \omega(n)}{n}$, then $G(n,p)$
a.a.s.~has a vertex of degree at most one.

An equivalent way of describing $G(n,p)$ is as the probability space
of graphs obtained by taking every edge of the complete graph $K_n$
independently with probability $p$. A variety of questions can be
asked when we consider a host graph $G$ other than $K_n$, and
consider the probability space of graphs obtained by taking every
edge of it independently with probability $p$ (we denote this
probability space as $G_p$).

The following question can be placed in this context and can be also
viewed as an attempt to understand the robustness of Dirac's
theorem. Let $G$ be a graph of minimum degree at least $\frac{n}{2}$
and note that $G$ is Hamiltonian by Dirac's theorem. Since
Hamiltonicity is a monotone graph property, we know that there
exists a threshold $p_0$ (see \cite{BoTh}) such that if $p \gg p_0$,
then $G_p$ is a.a.s.~Hamiltonian, and if $p \ll p_0$, then $G_p$ is
a.a.s.~not Hamiltonian. For random graphs, the threshold for
Hamiltonicity is $p_0 = \frac{\log n}{n}$ (it is moreover a sharp
threshold). What is the Hamiltonicity threshold for $G_p$, in
particular, does $G_p$ stay Hamiltonian for $p \ll 1$? Our main
theorem provides an answer to this question.

\begin{thm} \label{thm_mainrandom}
There exists a positive constant $C$ such that for $p\ge\frac{C\log
n}{n}$ and a graph $G$ on $n$ vertices of minimum degree at least
$\frac{n}{2}$, the random subgraph $G_{p}$ is a.a.s. Hamiltonian.
\end{thm}

This theorem establishes the correct order of magnitude of the
threshold function since if $p \le (1+o(1))\frac{\log n}{n}$, then
the graph a.a.s.~has isolated vertices. Also, since there are graphs
with minimum degree $\frac{n}{2}-1$ which are not even connected,
the minimum degree condition cannot be improved. Moreover, our
theorem can actually be viewed as an extension of Dirac's theorem
since the case $p=1$ is equivalent to Dirac's theorem.

\subsection{Hamiltonicity game}

Let $V$ be a set of elements and $\mathcal{F} \subseteq 2^{V}$ be a
family of subsets of $V$. A {\em Maker-Breaker game} involves two
players, named Maker and Breaker respectively, who alternately
occupy the elements of $V$, the {\em board} of the game. The game
ends when there are no unoccupied elements of $V$. Maker wins the
game if in the end, the vertices occupied by Maker contain as a
subset at least one of the sets in $\mathcal{F}$, the family of {\em
winning sets} of the game. Breaker wins otherwise.

Chv\'atal and Erd\H{o}s \cite{ChEr} were the first to consider
biased Maker-Breaker games on the edge set of the complete graph.
They realized that natural graph games are often ``easily'' won by
Maker when played fairly (that is when Maker and Breaker each claim
one element at a time). Thus for many graph games, it is natural to
give Breaker some advantage. In a $(1:b)$ Maker-Breaker game we
follow the same rule as above, but Maker claims one element each
round while Breaker claims $b$ elements each round. It is not too
difficult to see that Maker-Breaker games are {\em bias monotone}.
More specifically, if for some fixed game, Maker can win the $(1:b)$
game, then Maker can win the $(1:b')$ game for every $b'<b$. Thus it
is natural to consider the {\em critical bias} of a game, which is
defined as the maximum $b_0$ such that Maker wins the $(1:b_0)$
game.

One of the first biased games that Chv\'atal and Erd\H{o}s
considered in their paper was the Hamiltonicity game played on the
edge set of the complete graph. They proved that the $(1:1)$ game is
Maker's win, and that for any fixed positive $\varepsilon$ and $b(n)
\ge (1+\varepsilon)\frac{n}{\log n}$, the $(1:b)$ game is Breaker's
win for large enough $n$. They then conjectured that the critical
bias of this game should go to infinity as $n$ goes to infinity.
Bollob\'as and Papaioannou \cite{BoPa} verified their conjecture and
proved that the critical bias is at least $c \frac{\log n}{\log \log
n}$ for some constant $c>0$. Beck \cite{Beck2} improved on this
result by proving that the critical bias is at least $(\frac{\log
2}{27} - o(1))\frac{n}{\log n}$, thereby establishing the correct
order of magnitude of the critical bias. Krivelevich and Szab\'o
\cite{KrSz} further improved this result, and recently Krivelevich
\cite{Krivelevich} established the fact that the critical bias of
this game is asymptotically $\frac{n}{\log n}$. We refer the reader
to \cite{Beck4} for more information on Maker-Breaker games, and
general positional games.

In this context, and similarly to that of the question considered in
the previous subsection, we would like to strengthen Dirac's theorem
from the Maker-Breaker game point of view. Let $G$ be a graph of
minimum degree at least $\frac{n}{2}$ and consider the Hamiltonicity
Maker-Breaker game played on $G$ (note that $G$ is Hamiltonian by
Dirac's theorem). We can then ask the following questions: ``will
Maker win the $(1:1)$ game on any Dirac graph?'', and if so, then
``what is the largest bias $b$ such that Maker wins the $(1:b)$
game?''. In this paper we establish the threshold $b_0$ such that if
$b \ll b_0$, then Maker wins, and if $b \gg b_0$, then Breaker wins.

\begin{thm} \label{thm_maingame}
There exists a constant $c>0$ such that for $b\le\frac{cn}{\log n}$
and a graph $G$ on $n$ vertices of minimum degree at least
$\frac{n}{2}$, Maker has a winning strategy for the $(1:b)$
Maker-Breaker Hamiltonicity game on $G$.
\end{thm}

Our theorem implies that the critical bias of this game has order of
magnitude $\frac{n}{\log n}$ (note that the critical bias is at most
$(1+o(1))\frac{n}{\log n}$ by the result of Chv\'atal and Erd\H{o}s
mentioned above). Note that in this theorem, once all the elements
of the board are claimed, the edge density of Maker's graph is of
order of magnitude $\frac{\log n}{n}$ and this is the same as in
Theorem \ref{thm_mainrandom}. This suggests that as in many other
Maker-Breaker games, the ``probabilistic intuition'', a relation
between the critical bias and the threshold probability of random
graphs, holds here as well (see, \cite{Beck, Beck3}). In fact, this
is not a coincidence, and we will prove both theorems under one
unified framework.

\bigskip

\noindent \textbf{Notation}. A graph $G=(V,E)$ is given by a pair of
its vertex set $V=V(G)$ and edge set $E=E(G)$. We use $|G|$ or $|V|$
to denote the size of its vertex set. For a subset $X$ of vertices,
we use $e(X)$ to denote the number of edges within $X$, and for two
sets $X,Y$, we use $e(X,Y)$ to denote the number of pairs $(x,y)$
such that $x \in X, y\in Y$, and $\{x,y\}$ is an edge (note that $e(X,X)=2e(X)$). $G[X]$
denotes the subgraph of $G$ induced by a subset of vertices $X$. We
use $\overline{X}$ to denote the complement $V \setminus X$ of $X$,
and $N(X)$ to denote the collection of vertices which are adjacent
to some vertex of $X$. For two graphs $G_1$ and $G_2$ over the same
vertex set $V$, we define their intersection as $G_1 \cap G_2 = (V,
E(G_1)\cap E(G_2))$, their union as $G_1 \cup G_2 = (V, E(G_1) \cup
E(G_2))$, and their difference as $G_1 \setminus G_2 = (V, E(G_1)
\setminus E(G_2))$.

When there are several graphs under consideration, to avoid
ambiguity, we use subscripts such as $N_{G}(X)$ to indicate the
graph that we are currently interested in. We also use subscripts
with asymptotic notations to indicate dependency. For example,
$\Omega_\varepsilon$ will be used to indicate that the hidden
constant depends on $\varepsilon$. Throughout the paper, whenever we
refer, for example, to a function with subscript as $f_{3.1}$, we
mean the function $f$ defined in Lemma/Theorem 3.1. To simplify the
presentation, we often omit floor and ceiling signs whenever these
are not crucial and make no attempts to optimize absolute constants
involved. We also assume that the order $n$ of all graphs tends to
infinity and therefore is sufficiently large whenever necessary.
All logarithms will be in base $e \approx 2.718$.

\section{Dirac graphs}
\label{section_dirac} The following lemma used by S\'ark\"ozy and
Selkow \cite{SaSe}, and by Cuckler and Kahn \cite{CuKa}, classifies
Dirac graphs into three categories (a similar lemma has also been
used by Koml\'os, Sark\"ozy, and Szemer\'edi \cite{KoSaSz}). This
classification is an important tool in controlling Dirac graphs. A
\emph{half set} of a graph is a subset of the vertex set which has
size either $\lfloor\frac{n}{2}\rfloor$ or
$\lceil\frac{n}{2}\rceil$.

\begin{lem} \label{lem:dirac}
Let $\alpha\le\frac{1}{320}$ and $\gamma \le\frac{1}{10}$ be fixed
positive reals such that $\gamma \ge 32\alpha$. If $n$ is large
enough, then for every graph $G$ on $n$ vertices with minimum degree
at least $\frac{n}{2}$, one of the following holds:
\begin{enumerate}[(i)]
  \setlength{\itemsep}{1pt} \setlength{\parskip}{0pt}
  \setlength{\parsep}{0pt}
\item $e(A,B)\ge\alpha n^{2}$ for all half-sets $A$ and $B$ (not necessarily disjoint),
\item There exists a set $A$ of size $\frac{n}{2}\le|A|\le(\frac{1}{2}+16\alpha)n$
such that $e(A,\overline{A})\le6\alpha n^{2}$, and the induced subgraphs
on both $A$ and $\overline{A}$ have minimum degree at least $\frac{n}{5}$,
or
\item There exists a set $A$ of size $\frac{n}{2}\le|A|\le(\frac{1}{2}+16\alpha)n$
such that the bipartite graph induced by the edges between $A$
and $\overline{A}$ has at least $(\frac{1}{4}-14\alpha)n^{2}$
edges, and minimum degree at least $\frac{\gamma}{2}n$.
Moreover, either $|A|=\lceil\frac{n}{2}\rceil$, or the induced
subgraph $G[A]$ has maximum degree at most $\gamma
n$.\end{enumerate}
\end{lem}
\begin{proof}
Assume that property (i) does not hold, i.e., that $e(A,B)<\alpha
n^{2}$ for some two half-sets $A$ and $B$. For simplicity, we assume
that $n$ is even and $|A|=|B|=\frac{n}{2}$ (for odd $n$, some small
order terms will be added to the computation). Note that in this
case we have $|\overline{A\cup B}|=|A\cap B|$,  it follows that for
all $v \in A\cap B$, \[ |N(v)\cap(A\cup
B)|\ge|N(v)|-|\overline{A\cup B}|\ge\frac{n}{2}-|A\cap B|.\]
Therefore, \[ \alpha n^{2}>e(A,B)\ge \sum_{v\in A\cap
B}|N(v)\cap(A\cup B)|\ge |A\cap B|\cdot\left(\frac{n}{2}-|A\cap
B|\right).\] If $5\alpha n \le |A \cap B| \le (\frac{1}{2} -
5\alpha)n$, then the right hand side of the above estimate is at
least $5\alpha n \cdot (\frac{1}{2} - 5\alpha)n$, which by $\alpha
\le \frac{1}{320}$ is larger than $\alpha n^2$. Thus the above
inequality implies that either $|A\cap B|\le5\alpha n$ or $|A\cap
B|\ge(\frac{1}{2}-5\alpha)n$. We consider these two cases
separately.

\medskip
\noindent \textbf{Case 1} : $|A\cap B|\le5\alpha n$.

In this case, we have $|\overline{A} \setminus B|\le5\alpha n$ and
thus $e(A,\overline{A})\le e(A,B)+|\overline{A} \setminus B|\cdot
|A| \le \alpha n^2 + \frac{5}{2}\alpha n^{2} \le 4\alpha n^2$. Let
$A_{0}=A$ and $B_{0}=\overline{A}.$ Let $A_{1}$ be the set of
vertices of $A_{0}$ which have at most $\frac{n}{4}$ neighbors in
$A_{0}$, and note that by the minimum degree condition of the graph
$G$, every vertex in $A_{1}$ has at least $\frac{n}{4}$ neighbors in
$B_{0}$. Since $e(A, \overline{A}) \le 4\alpha n^2$, we have
$|A_{1}|\le 16\alpha n$. Similarly, we can define $B_{1}$ so that
$|B_{1}|\le 16\alpha n$. Let $A'=(A_{0}\setminus A_{1})\cup B_{1}$
and $B'=(B_{0}\setminus B_{1})\cup A_{1}$. We obtain a new partition
$V=A'\cup B'$ such that
\[
\left(\frac{1}{2}-16\alpha\right)n\le|A'|,|B'|\le\left(\frac{1}{2}+16\alpha\right)n\]
and the minimum degree inside each part is at least
$\frac{n}{4}-16\alpha n\ge\frac{n}{5}$. Moreover, we have
\begin{eqnarray*}
e(A',B') & = & e(A_{0},B_{0})-e(A_{1},B_{0} \setminus B_{1})+e(A_{1}, A_{0} \setminus A_{1})
                             -e(B_{1},A_{0} \setminus A_{1})+e(B_{1}, B_{0} \setminus B_{1})\\
         & \le & e(A_{0},B_{0})-e(A_{1},B_{0})+e(A_{1}, A_{0})
                             -e(B_{1},A_{0})+e(B_{1}, B_{0}) + 2e(A_1,B_1),
\end{eqnarray*}
which by $e(A_1,B_0) \ge e(A_1,A_0)$ and $e(B_1,A_0) \ge e(B_1,B_0)$
gives
\begin{eqnarray*}
e(A',B') \le e(A_{0},B_{0}) + 2e(A_1,B_1) \le 4\alpha n^{2}+2|A_{1}||B_{1}|
   \le 4\alpha n^2+2^9\alpha^2 n^2 \le  6\alpha n^{2}.
\end{eqnarray*}
Since $|A'\cup B'| = n$,
the larger set among $A'$ and $B'$ has size at least $\frac{n}{2}$,
and it satisfies property (ii).

\medskip
\noindent \textbf{Case 2} : $|A\cap B|\ge(\frac{1}{2}-5\alpha)n$.

In this case, we have $|A \setminus B|\le5\alpha n$, and therefore
$e(A,A)\le e(A,B)+|A \setminus B|\cdot |A| \le \alpha n^2 +
\frac{5}{2}\alpha n^2 \le 4\alpha n^{2}$. By the minimum degree
condition, we have
$e(A,\overline{A})\ge|A|\cdot\frac{n}{2}-e(A,A)\ge(\frac{1}{4}-4\alpha)n^{2}.$
Let $A_{0}=A$ and $B_{0}=\overline{A}$. Let $A_{1}$ be the set of
vertices of $A_{0}$ which have at most $\frac{n}{4}$ neighbors in
$B_{0}$. Note that by the minimum degree condition of the graph $G$,
every vertex in $A_{1}$ has at least $\frac{n}{4}$ neighbors in
$A_{0}$. Also, by the estimate \[
\left(\frac{1}{4}-4\alpha\right)n^{2}\le
e(A_0,B_0)\le\frac{n}{4}\cdot|A_{1}|+\frac{n}{2}\cdot\left(\frac{n}{2}-|A_{1}|\right)=\frac{n^{2}}{4}-\frac{n}{4}\cdot|A_{1}|,\]
we have $|A_{1}|\le 16\alpha n$. Similarly define $B_{1}$ so that
$|B_{1}|\le 16\alpha n$. Let $A'=(A_{0}\setminus A_{1})\cup B_{1}$
and $B'=(B_{0}\setminus B_{1})\cup A_{1}$. We have
\begin{eqnarray*}
e(A',B') & = & e(A_{0},B_{0})-e(A_{1},B_{0} \setminus B_{1})+e(A_{1}, A_{0} \setminus A_{1})
                             -e(B_{1},A_{0} \setminus A_{1})+e(B_{1}, B_{0} \setminus B_{1})\\
         & \ge & e(A_{0},B_{0})-e(A_{1},B_{0})+e(A_{1}, A_{0})
                             -e(B_{1},A_{0})+e(B_{1}, B_{0}) - 2e(A_1,B_1),
\end{eqnarray*}
which by $e(A_1,B_0) \le e(A_1,A_0)$ and $e(B_1,A_0) \le e(B_1,B_0)$
gives
\begin{align*}
e(A',B') &\ge e(A_{0},B_{0}) - 2e(A_1,B_1)
 \ge \left(\frac{1}{4}-4\alpha\right)n^{2}-2|A_1||B_1| \\
 &\ge \left(\frac{1}{4}-4\alpha\right)n^2 - 2^9\alpha^2n^2
 \ge \left(\frac{1}{4}-6\alpha\right)n^{2},
 \end{align*}
and all vertices in $A'$ have at least $\frac{n}{4}-16\alpha
n\ge\frac{n}{5}$ neighbors in $B'$ (and vice versa).

Since $|A' \cup B'|= n$, we may assume without loss of generality
that $|A'|\ge\lceil\frac{n}{2}\rceil$. Note that $|A'| \le |A_0| +
|B_1| \le \left(\frac{1}{2}+16\alpha\right)n$. If
$|A'|=\lceil\frac{n}{2}\rceil$ or $G[A']$ has maximum degree at most
$\gamma n$, then we have already found our set $A'$. Otherwise, move
a vertex in $A'$ which has at least $\gamma n$ neighbors in $A'$ to
the other side $B'$, and update $A',B'$ accordingly. Repeat this
until we reach the point where $|A'|=\lceil\frac{n}{2}\rceil$ or
$G[A']$ has maximum degree at most $\gamma n$. Since we moved at
most $16\alpha n$ vertices from $A'$ to $B'$, all the vertices in
$B'$ have at least $\gamma n - 16\alpha n \ge \frac{\gamma}{2}n$
neighbors in $A'$. On the other hand, all the remaining vertices in
$A'$ still have at least $\frac{n}{5}\geq \frac{\gamma}{2}n$
neighbors in $B'$. Finally, $e(A',B')$ decreases by at most
$16\alpha n\cdot\frac{n}{2}=8\alpha n^{2}$ and is still at least
$(\frac{1}{4} - 14\alpha)n^2$. Thus we found our claimed set as in
property (iii).
\end{proof}

\section{Rotation and extension}

We will prove our two main theorems under one general framework
provided in this section.
Our main tool is P\'osa's rotation-extension technique
which first appeared in \cite{Posa} (see also
\cite[Ch. 10, Problem 20]{Lovasz}). We start by briefly
discussing this powerful tool,
which exploits the expansion property of the graph.

\begin{figure}[b]
  \centering
    \begin{tikzpicture}

\def\xendtwo{8 cm}
\def\yshifttwo{2.2 cm}

   \foreach \x in {0.0,0.8,...,8.8} {
            \draw [fill=black] (\x cm, \yshifttwo) circle (0.5mm);
        }

   \foreach \x in {0.0,0.8,1.6} {
            \draw (\x cm,\yshifttwo) -- (\x cm + 0.8cm,\yshifttwo);
        }
   \draw [dotted] (2.4 cm,\yshifttwo) -- (3.2cm,\yshifttwo);

   \foreach \x in {3.2,4.0,...,8.0} {
            \draw (\x cm,\yshifttwo) -- (\x cm + 0.8cm,\yshifttwo);
        }


    \draw (\xendtwo, \yshifttwo) .. controls (6cm,\yshifttwo + 1.1cm)
            and (4.4cm,\yshifttwo + 1.1cm) .. (2.4cm, \yshifttwo);

    \draw (3.3cm, \yshifttwo - 0.4cm) node {$v_{i+1}$};
    \draw (2.46cm, \yshifttwo - 0.4cm) node {$v_{i}$};

    \draw (0.0cm, \yshifttwo - 0.4cm) node {$v_{0}$};
    \draw (8.0cm, \yshifttwo - 0.4cm) node {$v_{\ell}$};

\end{tikzpicture}
  \caption{Rotating a path.}
  \label{fig_rotation}
\end{figure}
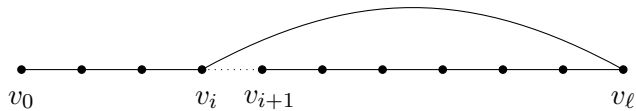

Let $G$ be a connected graph and let $P = (v_0, \cdots, v_\ell)$ be
a path on some subset of vertices of $G$ ($P$ is not necessarily a
subgraph of $G$). If $\{v_0, v_\ell\}$ is an edge of the graph, then
we can use it to close $P$ into a cycle. Since $G$ is connected,
either the graph $G \cup P$ is Hamiltonian, or there exists a longer
path in this graph. In the second case, we say that we
\emph{extended} the path $P$.

Assume that we cannot directly extend $P$ as above, and assume that
$G$ contains an edge of the form $\{v_\ell, v_i\}$ for some $i$.
Then $P' = (v_0, \cdots, v_i, v_\ell, v_{\ell-1}, \cdots, v_{i+1})$
forms another path of length $\ell$ in $G \cup P$ (see Figure
\ref{fig_rotation}). We say that $P'$ is obtained from $P$ by a
\emph{rotation} with {\em fixed endpoint} $v_0$, \emph{pivot point}
$v_i$, and {\em broken edge} $(v_i, v_{i+1})$. Note that after
performing this rotation, we can now close a cycle of length $\ell$
also using the edge $\{v_0, v_{i+1}\}$ if it exists in $G \cup P$.
As we perform more and more rotations, we will get more such
candidate edges (call them \emph{closing edges}). The
rotation-extension technique is employed by repeatedly rotating the
path until one can find a closing edge in the graph, thereby
extending the path.

Let $P''$ be a path obtained from $P$ by several rounds of
rotations. An important observation that we will use later is that
for every interval $I = (v_j, \cdots, v_k)$ of vertices of $P$ ($1
\le j < k \le \ell$), if no edges of $I$ were broken during these
rotations, then $I$ appears in $P''$ either exactly as it does in
$P$, or in the reversed order. We define the {\em orientation}, or
{\em direction}, of a path $P''$ with respect to an interval $I$ to
be {\em positive} in the former situation, and {\em negative} in the
latter situation.

\medskip

We will use rotations and extensions as described above to prove our
main theorem. The main technical twist is to split the given graph
into two graphs, where the first graph will be used to perform
rotations, and the second graph to perform extensions. Similar
ideas, such as \emph{sprinkling}, have been used in proving many
results on Hamiltonicity of random graphs. The one  closest to our
implementation appears in the recent paper of Ben-Shimon,
Krivelevich, and Sudakov \cite{BeKrSu}.

\subsection{Rotation-Extension for general graphs}

In this subsection, we develop a framework useful in tackling the
first and the second cases of Lemma \ref{lem:dirac}. We assume that
all the graphs appearing in this subsection are defined over a fixed
vertex set $V$ of size $n$ (therefore if there are several graphs
under consideration, then they share the same vertex set). We first
specify the roles of the graphs performing rotations and extensions.
\begin{defn}
\label{def:rotation}Let $\xi$ be a positive constant. We say
that a graph $G$ has property $\mathcal{RE}(\xi)$
if it is connected, and for every path $P$ with a fixed edge $e$, (i)
there exists a path containing $e$ longer than $P$ in the graph $G \cup P$, or
(ii) there exists a set of vertices $S_{P}$
of size $|S_{P}|\ge\xi n$ such that for every vertex
$v\in S_{P}$, there exists a set $T_{v}$ of size $|T_{v}|\ge\xi n$
such that for every $w\in T_{v}$, there exists a path
containing $e$ of the same length as $P$ that starts at $v$, and ends at $w$.
\end{defn}

\begin{defn}
\label{def:extension}Let $\xi$ be a positive constant and let
$G_{1}$ be a graph with property $\mathcal{RE}(\xi)$. We say that a
graph $G_{2}$ \emph{complements} $G_{1}$, if for every path $P$ with
a fixed edge $e$, (i) there exists a path containing $e$ longer than
$P$ in the graph $G_1 \cup P$, or (ii) there exist $v\in S_{P}$ and
$w\in T_{v}$, such that $\{v,w\}$ is an edge of $G_{1} \cup G_{2}
\cup P$ (the sets $S_{P}$ and $T_{v}$ are as defined in Definition
\ref{def:rotation}).
\end{defn}

The next proposition asserts that two graphs as in the above two
definitions together give Hamiltonicity. In fact, we will obtain a
slightly stronger property which is called {\em Hamilton
connectivity}. A graph is said to be {\em Hamilton connected} if for
every pair of vertices $x$ and $y$, there exists a path of length
$n-1$ that has $x$ and $y$ as its two endpoints. Since a Hamilton
connected graph is necessarily non-empty, by taking $x$ and $y$ as
two endpoints of an edge in the graph, we can see that Hamilton
connectivity implies Hamiltonicity.
\begin{prop} \label{prop:rotationextension}
Let $\xi$ be a positive constant. If $G_{1}\in\mathcal{RE}(\xi)$ and
$G_{2}$ complements $G_{1}$, then $G_{1}\cup G_{2}$ is  Hamilton
connected.
\end{prop}
\begin{proof}
Let $v_{1}$ and $v_{2}$ be two arbitrary vertices. If
$\{v_{1},v_{2}\}$ is not an edge of $G_{1}\cup G_{2}$, then let
$G_2'$ be the graph obtained by adding the edge $e =
\{v_{1},v_{2}\}$ to the graph $G_2$, otherwise let $G_2' = G_2$.
Note that $G_2'$ complements $G_1$. Let $P$ be a longest path in
$G_1 \cup G_2'$ which contains $e$ (say it has length $\ell$). By
Definition \ref{def:extension}, there exist vertices $v'\in S_{P}$
and $w'\in T_{v'}$ such that $\{v',w'\}$ is an edge of $G_{1}\cup
G_{2}'$ (where the sets $S_{P}$ and $T_{v}$ are as in Definition
\ref{def:rotation}). Thus we can find a cycle containing $e$, of
length $\ell$.

If this cycle is not a Hamilton cycle, then by the connectivity of
$G_{1}$, there exists a vertex $x$ not in the cycle, which is
incident to some vertex of the cycle. There are two ways to
construct a path of length $\ell +1$ using this edge and the cycle,
and one of them necessarily contains the edge $e$. Since this
contradicts the maximality of $P$, the cycle must have been a
Hamilton cycle. By removing the edge $e$, we get a Hamilton path
connecting $v_1$ and $v_2$ in $G_1 \cup G_2$.
\end{proof}

Thus our strategy for proving Hamiltonicity is to find a subgraph
with property $\mathcal{RE}(\xi)$ and a one that complements it. In
the remainder of the subsection, we provide a list of deterministic
properties, which when satisfied, imply property
$\mathcal{RE}(\xi)$. After establishing this lemma, later it will
suffice to verify that these deterministic properties hold for the
graphs  we are interested in.

\begin{defn} \label{def:expander}
Let $\varepsilon$ and $r$ be positive reals.
A graph $G$ is a {\em half-expander
with parameters $\varepsilon$ and $r$}, if the following properties hold.
\begin{enumerate}[(i)]
  \setlength{\itemsep}{1pt} \setlength{\parskip}{0pt}
  \setlength{\parsep}{0pt}
\item For every set $X$ of vertices of size $|X|\le \frac{\varepsilon n}{r}$, $|N(X)|\ge r|X|$,
\item for every set $X$ of vertices of size $|X|\ge \frac{n}{\varepsilon r}$, $|N(X)|\ge(\frac{1}{2}-\varepsilon)n$, and
\item for every pair of disjoint sets $X,Y$ such that $|X|,|Y|\ge(\frac{1}{2}-\varepsilon^{1/5})n$,
$e(X,Y)>2n$.
\end{enumerate}
\end{defn}

\begin{lem} \label{lem_halfexpander}
There exists a positive $\varepsilon_0$ such that for every positive
$\varepsilon \le \varepsilon_0$, the following holds for every $r
\ge 16\varepsilon^{-3} \log n$: every half-expander
with parameters $\varepsilon$ and $r$ has property
$\mathcal{RE}(\frac{1}{2}+\varepsilon)$.
\end{lem}
\begin{proof}
For simplicity of notation, we assume that we are given a
half-expander with parameters $\varepsilon^{4}$ and $r$, and will
prove that it has property
$\mathcal{RE}(\frac{1}{2}+\varepsilon^{4})$. Let $\varepsilon_0 =
25^{-5}$, and suppose that we are given positive reals $\varepsilon
\le \varepsilon_0$ and $r \ge 16\varepsilon^{-12}\log n$. Let $G$ be
a half-expander with parameters $\varepsilon^4$ and $r$. To prove
that $G$ is connected, take two vertices $v$ and $w$. By Properties
(i) and (ii) of Definition \ref{def:expander}, there exist sets
$A_v$ and $A_w$ each of size at least $(\frac{1}{2}-\varepsilon^4)n
\ge (\frac{1}{2} - \varepsilon^{4/5})n$ such that the connected
component containing $v$ contains $A_v$, and the connected component
containing $w$ contains $A_w$. There exists an edge between $A_v$
and $A_w$ by Property (iii) of Definition \ref{def:expander}.
Consequently, there exists a path between $v$ and $w$ in $G$ for all
pairs of vertices $v$ and $w$.

Let $P=(v_{0},\cdots,v_{\ell})$ be a path with some fixed edge $e_f
= \{v_f, v_{f+1}\}$, and let $F$ be the set $\{v_{f-1}, v_f,
v_{f+1}, v_{f+2}\}$. If there is a path longer than $P$ that
contains $e_f$ in the graph $G \cup P$, then there is nothing to
prove since it satisfies the first condition of Definition
\ref{def:rotation}. Thus we may assume that $P$ is a longest path in
$G \cup P$ that contains $e_f$. We start by rotating $v_{0}$ to
construct the set $S_{P}$ as in Definition \ref{def:rotation}.
Afterwards, we will construct the sets $T_v$ the same way.

For a subset $X=\{v_{a_{1}},v_{a_{2}}\cdots,v_{a_{i}}\}$ of vertices
of $P$, let $X^{-}=\{v_{a_{1}-1},v_{a_{2}-1},\cdots,v_{a_{i}-1}\}$
and $X^{+}=\{v_{a_{1}+1},v_{a_{2}+1},\cdots,v_{a_{i}+1}\}$ (if the
index becomes either $-1$ or $\ell+1$, then remove the corresponding
vertex from the set $X^{-}$ or $X^{+}$). Throughout the proof we
will repeatedly consider the operation $X^+$ and $X^-$ for various
sets $X$. While performing this operation, one must take special
care of the vertices which lie in the boundary of the intervals
$P_i$. However, we will ignore the effect of the boundary vertices,
since it will simplify the computation, and will only affect it by
some small order terms. Let $k= 4\varepsilon^{-4}\log n$, and
partition the path $P$ into $k$ consecutive intervals of lengths as
equal as possible. Denote these intervals as $P_{1},\cdots,P_{k}$.

\medskip
\noindent {\bf Step 1}: {\em Initial rotations.}
\medskip

Our argument is based on that of Sudakov and Vu \cite{SuVu} where
one performs rotations and extensions in a very controlled manner.
Let $S_{0}=\{v_{0}\}$. We will iteratively construct sets $S_{i}$
for $i\ge0$ so that $|S_{i}|=\varepsilon^{-8i}$, and for all $v\in S_{i}$, there
exists a path of length $\ell$ which starts at $v$, ends at
$v_{\ell}$, contains $e_f$, and has been obtained from $P$ by $i$
rounds of rotations. We will continue to construct sets as long as
$|S_{i}|\le \frac{\varepsilon^4 n}{r}$. Note that this implies $i\le\log n$.
For a vertex $v\in S_{i}$, let $e_{v,1},\cdots,e_{v,i}$ be the
broken edges created in constructing the path from $v$ to
$v_{\ell}$, in the order they were broken (we call them the {\em
broken edges} of $v$). Note that the order in which the broken edges
were created is not necessarily the same as the order in which they
appear along the path. We impose the following conditions on $S_i$:

\medskip
\noindent {\bf (Universality)} for all $1 \le a \le i$, there exists
an index $j_a$ such that for all $v \in S_i$, the edge $e_{v,a}$
belongs to $P_{j_a}$. Moreover, if several broken edges belong to
the same interval, then the order in which they appear within the
interval also does not depend on $v$.
\medskip

 An important byproduct of this property is that for every
interval $P_{j}$ which does not contain a broken edge, there is a
fixed orientation so that for all $v \in S_i$, the path from $v$ to
$v_\ell$ traverses $P_{j}$ in this orientation. Moreover, the order
in which each non-broken interval appears along these paths does not
depend on $v$ (thus is {\em universal}, see Figure
\ref{fig_generalpath}).

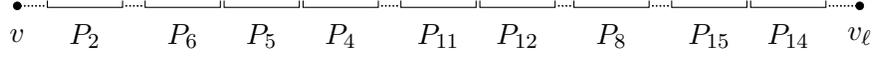
\begin{figure}[t]
  \centering
    \begin{tikzpicture}

\def\yshifttwo{0 cm}

   \draw [fill=black] (0,0) circle (0.5mm);
   \draw (0.0cm, \yshifttwo - 0.4cm) node {$v$};
   \draw [fill=black] (112mm,0mm) circle (0.5mm);
   \draw (112mm, \yshifttwo - 4mm) node {$v_{\ell}$};

   \foreach \x/\y/\z in {4/14/P_{2}, 17/27/P_{6}, 27.5/37.5/P_{5}, 38/48/P_{4}, 51/61/P_{11}, 61.5/71.5/P_{12}, 74/84/P_{8}, 87/97/P_{15}, 97.5/107.5/P_{14}} {
     \draw (\x mm, 0.7 mm) -- (\x mm,-0.2mm) -- ( \y mm,-0.2mm) --(\y mm, 0.7mm);
     \draw (\x mm + 5 mm, \yshifttwo - 4mm) node {$\z$};
   }

   \foreach \x in {0.5,1.0,...,3.5} {
            \draw [fill=black] (\x mm, \yshifttwo) circle (0.05mm);
   }
   \foreach \x in {14.5,15.0,...,16.5} {
            \draw [fill=black] (\x mm, \yshifttwo) circle (0.05mm);
   }
   \foreach \x in {48.5,49.0,...,50.5} {
            \draw [fill=black] (\x mm, \yshifttwo) circle (0.05mm);
   }
   \foreach \x in {72,72.5,...,73.5} {
            \draw [fill=black] (\x mm, \yshifttwo) circle (0.05mm);
   }
   \foreach \x in {84.5,85.0,...,86.5} {
            \draw [fill=black] (\x mm, \yshifttwo) circle (0.05mm);
   }
   \foreach \x in {108,108.5,...,111.5} {
            \draw [fill=black] (\x mm, \yshifttwo) circle (0.05mm);
   }
   
   \draw [color=white] (0, 10mm) circle (0.02mm);

\end{tikzpicture}
  \caption{A path obtained by several rounds of rotations. The order of the intervals
have been changed as a consequence of these rotations. Intervals that do
not appear in the figure ($P_1, P_3, P_7, P_9, P_{10}, P_{13}, P_{16}, \cdots$) contain broken edges,
and the vertices in those intervals will be spread out among the dotted area.
Note that non-broken intervals appear either in the original order, or in the reversed order.
By universality, even if we change the vertex $v$ into another vertex,
the order of the non-broken intervals $P_2, P_6, P_5, \cdots$ as above will not change.}
  \label{fig_generalpath}
\end{figure}

Assume we have completed constructing the set $S_{i}$ which has the
properties listed above. Let $S_{i+1}^0$ be the set $N(S_i)
\setminus (F \cup \bigcup_{a=1}^{i} (S_{a} \cup S_{a}^- \cup
S_{a}^+))$. If there is a vertex in $S_{i+1}^0$ which is not in the
path $P$, then we can use it to find a path longer than $P$ that
contains the edge $e_f$. Therefore we may assume that $S_{i+1}^0
\subset V(P)$. Since we removed all the vertices belonging to $S_a,
S_a^-, S_a^+$ for $a \le i$ when defining $S_{i+1}^0$, all the
vertices in $S_{i+1}^0$ can be used as pivot points to create new
endpoints (note that the broken edges obtained by this procedure are
necessarily distinct from all the previous broken edges). Since
$|N(S_i)| \ge r |S_i|$, we have the following estimate on the size
of $S_{i+1}^0$:
\begin{align}
|S_{i+1}^0| \ge r|S_{i}|-3\sum_{a=1}^{i}|S_{a}| - 4
\ge r \varepsilon^{-8i} - 3 \varepsilon^{-8(i+1)} - 4 \ge \left(\frac{r}{2}\right) \varepsilon^{-8i}.  \label{eqn:initialrotation}
\end{align}
It now suffices to choose a suitable subset of $S_{i+1}^{0}$ which also
satisfies universality.

Pick an arbitrary $v\in S_{i}$, and for each $P_{j}$, let $w(j)$ be
the number of broken edges of $v$ that $P_{j}$ contains (note that
by the universality, the choice of $v$ does not matter). Note that
\[ \sum_{j=1}^{k}(w(j)+1) = i + k \le \left(\frac{4}{\varepsilon^{4}}+1 \right)\log n 
\le \frac{\varepsilon^8 r}{2}. \]
Consequently, there is an index $j_*\in[k]$ such that $P_{j_*}$
contains at least $\frac{2(w(j_*)+1)}{\varepsilon^8r}$ proportion of the vertices
of the set $S_{i+1}^0$. In other words we have
\[ |P_{j_*}\cap S_{i+1}^0| \ge \left( \frac{r}{2} \right) \varepsilon^{-8i} \cdot \frac{2(w(j_*)+1)}{\varepsilon^8 r}
= \varepsilon^{-8(i+1)}  \cdot (w(j_*)+1)).\]
By using the vertices in $P_{j_*}\cap S_{i+1}^0$ as pivot points, we
can obtain a set of new endpoints $S_{i+1}'$ with
$|S_{i+1}'|\ge\varepsilon^{-8(i+1)}\cdot(w(j_*)+1)$. By construction, all the
newly added pivot points and corresponding broken edges belong to
the same interval $P_{j_*}$. Therefore, it suffices to find a large
subset $S_{i+1}$ of $S_{i+1}'$ such that the broken edges of these
vertices that belong to $P_{j_*}$ appear in some universal order
(note that this automatically is true for indices other than $j_*$
by the same property for $S_{i}$).

By definition, for $h = w(j_*)$, there exist indices $i_1, i_2,
\cdots, i_h$ of value at most $i$ such that for every $v\in
S_{i+1}'$, the broken edges $e_{v,i_1},\cdots,e_{v,i_h},e_{v,i+1}$
are in $P_{j_*}$ ($e_{v,i+1}$ is the newly created broken edge in
the $(i+1)$-th round). By the hypothesis, we know that
$e_{v,i_1},\cdots,e_{v,i_h}$ appear in some fixed order which does
not depend on $v$. There are $h+1$ relative positions that
$e_{v,i+1}$ can lie within that ordering. We let $S_{i+1}$ be a
subset of $S_{i+1}'$ of size at least
$\frac{|S_{i+1}'|}{h+1}=\frac{|S_{i+1}'|}{w(j_*)+1}$ such that for
all the vertices in this set, the new broken edge has the same
relative order in $P_{j_*}$ with respect to the edges
$e_{v,i_1},\cdots,e_{v,i_h}$. This choice of $S_{i+1}$ satisfies all
our assumptions, and we have $|S_{i+1}|\ge \varepsilon^{-8(i+1)}$. Redefine
$S_{i+1}$ as an arbitrary subset of size exactly $\varepsilon^{-8(i+1)}$. Repeat
the above until we have a set $S_{t-1}$ of size at least
$\frac{\varepsilon^4 n}{r}$ and redefine $S_{t-1}$ as an arbitrary subset
of size exactly $\frac{\varepsilon^4 n}{r}$. Repeat the above process 
one more time to obtain a set $S_{t}$ of
size exactly $\frac{n}{\varepsilon^4 r}$ (note that $t \le \log n$).

\medskip
\noindent {\bf Step 2}: {\em Terminal rotations.}
\medskip

Let $k'=k-t-2$. There are at most $t+2$ intervals which contain at
least one broken edge for some vertex of $S_{t}$, or intersects $F$,
and thus at least $k'$ intervals do not have this property. For
notational convenience, relabel the intervals so that the intervals
$P_{1},\cdots,P_{k'}$ contain no broken edges and do not intersect
$F$, and let $P' = P_1 \cup \cdots \cup P_{k'}$. Note that
\[ |V(P) \setminus P'| \le \left(t + 2\right) \left(\frac{n}{k} + 1 \right) \le (\log n + 2) \left(\frac{n}{k} + 1\right)
\le \frac{4n \log n}{k} \le \varepsilon^4 n. \] Further assume that
for $1\le i\le k'$, each path from $v\in S_{t}$ to $v_{\ell}$
traverses the interval $P_{i}$ positively (we lose some generality
here, but we use no properties of this special case, and the
assumption is made just for the sake of clarity of presentation).
Define $S_{P}$ as the collection of vertices $v \in P$ which have
the property that in $G \cup P$ there exists a path of length $\ell$
containing $e_f$ that starts at $v$ and ends at $v_{\ell}$. Note
that $S_P$ contains $S_t$.

We want to show that $|S_{P}|\ge(\frac{1}{2}+\varepsilon^{4})n$.
Assume to the contrary that
$|S_{P}|<(\frac{1}{2}+\varepsilon^{4})n$. We claim that under this
assumption, the inequality $|S_{P}^{+}\Delta S_{P}^{-}|\le
22\varepsilon n$ holds. The proof of this claim will be given later
(see Claim \ref{cla:plusandminus}). Given this claim, consider the
set $Z=P' \setminus (S_P^+ \cup S_P^-)$. We have $|S_{P}^{+}\cup
S_{P}^{-}|\le|S_{P}|+22\varepsilon
n\le(\frac{1}{2}+23\varepsilon)n$. Since $|V(P)\setminus P'|\le \varepsilon^4 n$,
this implies $|Z|\ge |V(P)| - (\frac{1}{2}+23\varepsilon)n - \varepsilon^4 n
\ge |V(P)| - (\frac{1}{2}+24\varepsilon)n$, and for the set
$Z' := Z \cap (P')^- \cap (P')^+$, we have $|Z'| \ge |Z| - 2k' \ge
|V(P)| - (\frac{1}{2}+25\varepsilon)n$. Note that if some vertex $v \in S_P$
is adjacent to some vertex $w \in Z'$ and both $\{w^-,w\}$, $\{w,
w^+\}$ have not been broken while obtaining $v$ as an endpoint, then
we obtain a contradiction since this necessarily gives $w^-$ or
$w^+$ as a new endpoint, which by the definition of $Z'$ is not in
$S_P$.

Let $Y=N(S_{t})\cap P'$. Since the path from $v\in S_{t}$ to
$v_{\ell}$ traverses the intervals $P_{i}$ ($1\le i\le k'$)
positively, we can use the vertices of $Y$ as pivot points to
construct endpoints $Y^{-}$. We have
$|Y^{-}|\ge|N(S_{t})|-(\varepsilon^4 n + k)
\ge(\frac{1}{2}-3\varepsilon)n$. For each vertex $y \in Y^{-}$, fix
one path of length $\ell$ which starts at $y$, ends at $v_\ell$, and
has the property that all the broken edges but the last one lie
outside of $P'$. Thus at most one broken edge will lie inside $P'$.
Consequently, if some vertex $y \in Y^-$ has at least three
neighbors in $Z'$, then we necessarily have a vertex $w \in Z'$ for
which both $\{w, w^-\}$ and $\{w, w^+\}$ are not broken edges of
$y$. Then by the observation made in the previous paragraph we reach
a contradiction. Furthermore, by the maximality of $P$, we know that
there are no edges between $Y^{-}$ and $V \setminus V(P)$. Therefore,
for the set $Z'' = Z' \cup (V \setminus V(P))$ which is of size
$|Z''| \ge (|V(P)| - (\frac{1}{2}+25\varepsilon)n) + (|V| - |V(P)|)
\ge (\frac{1}{2}-25\varepsilon)n$, there are no edges between
$Y^{-}$ and $Z''$.
However, since $|Y^{-}|$ and $|Z''|$ are both at least
$(\frac{1}{2}-25\varepsilon)n$, by property (iii) of Definition
\ref{def:expander}, there exist more than $2n$ edges between $Y^{-}$
and $Z''$ and therefore some vertex in $Y^-$ must have at least three
neighbors in $Z''$. Consequently, we must have had
$|S_{P}|\ge(\frac{1}{2}+\varepsilon^{4})n$.

\medskip
\noindent {\bf Step 3}: {\em Rotating the other endpoint.}
\medskip

For every $v\in S_P$, there exists a path containing $e_f$ of length
$\ell$ which starts at $v$ and ends at $v_\ell$. Now by repeating
the above for the other endpoint $v_\ell$, we can see that for every
$v \in S_P$, there exists a set $T_v$ of size at least $(\frac{1}{2}
+ \varepsilon^4)n$ such that for every $w \in T_v$, there exists a
path of length $\ell$ which starts at $v$, ends at $w$, and contains
the edge $e_f$.
\end{proof}

It remains to prove the claim. The intuition behind this perhaps
strangely looking claim comes from the following two non-Hamiltonian
graphs whose minimum degrees are slightly less than $\frac{n}{2}$.
First, consider the graph consisting of two disjoint cliques of size
$\frac{n}{2}$ connected by a single edge, and consider a Hamiltonian
path in it. It is not too difficult to see that by rotating the
starting point we only get the first half of the path as new
starting points. More precisely, using the same notation as in the
proof above, we will get $|S_P| = \frac{n}{2} - 1$ and $|S_P^+
\Delta S_P^-| = 3$. Second, consider a complete bipartite graph in
which one part $A$ has one more vertex than the other part $B$, and
consider a Hamiltonian path in it (it must be an $A$-$A$ path). In
this case, by rotating the starting point we only get the vertices
in $A$ as new starting points, and therefore $|S_P| = \frac{n}{2}$
and $|S_P^+ \Delta S_P^-| = 0$. Note that the two graphs above both
have $|S_P|$ close to $\frac{n}{2}$ and $|S_P^+ \Delta S_P^-|$
small, but for very different reasons. Our claim asserts that, in
general, if the given graph has $|S_P|$ close to $\frac{n}{2}$, then
it indeed is true that the graph has small $|S_P^+ \Delta S_P^-|$.

\begin{claim} \label{cla:plusandminus}
If $|S_P|<(\frac{1}{2}+\varepsilon^4)n$, then $|S_{P}^{+}\Delta S_{P}^{-}|\le 22\varepsilon n$.
\end{claim}
\begin{proof}
Recall that we will ignore the effect of the boundary vertices while
performing the operations $X^-$ and $X^+$ for sets $X$, since it
will simplify the computation, and will only affect it by some small
order terms. The main strategy is as following. We first rotate the
path $P$ in two ways to obtain some set $Q$ of endpoints in two
different ways while keeping a big chunk $P''$ of $P$ not broken.
For each endpoint $w$ in $Q$, the two paths will traverse $P''$ in
opposite direction (see Figure \ref{fig_twicerotate}). If this is
the case, then both sets $(N(Q) \cap P'')^-$ and $(N(Q) \cap P'')^+$
become subsets of $S_P$. From this we will conclude that the two
sets $S_P^+$ and $S_P^-$ do not differ too much.

We follow the same notation as in the proof above. Recall that the
set $S_P$ was defined as the collection of vertices $v$ in $P$ for
which there exists a path of length $\ell$ starting at $v$ and
ending at $v_\ell$ that contains some fixed edge $e_f$. Recall that
$P' = P_1 \cup \cdots \cup P_{k'}$, $|V(P)\setminus P'| \le
\varepsilon^4 n$, and that by property $(ii)$ of Definition
\ref{def:expander}, for every set $X$ of size $|X| \ge
\frac{n}{\varepsilon^4 r}$, we have $|N(X)| \ge (\frac{1}{2} -
\varepsilon^4)n$. Note that by using the vertices in $N(S_t) \cap
P'$ as pivot points, we get $|S_P \cap P'| \ge |N(S_t) \cap P'| \ge
|N(S_t)| - |V(P) \setminus P'| \ge (\frac{1}{2} - 2\varepsilon^4)n$.
For a subset $I$ of $[k']$, we define $P_I = \cup_{i \in I} P_i$.

We first make a simple observation. Let $X$ be some set of endpoints
obtained by rotating the given path $P$, where $|X| \ge
\frac{n}{\varepsilon^4 r}$. Our rotations have been carefully performed,
hence in the next round of rotation, many vertices in $N(X)$ will
give rise to vertices in $S_P$. Thus if $|N(X)|$ is close to
$\frac{n}{2}$, then since $|S_P| < \left(\frac{1}{2} +
\varepsilon^4\right)n$, we will recover most of the vertices in
$S_P$ in the next round of rotation. The following simple
proposition formalizes this intuition and will be used several times
in proving the claim.

\begin{prop} \label{prop:keyprop}
Let $X$ be a subset of $S_P$ of size $|X| \ge \frac{n}{\varepsilon^4 r}$,
and for every $v \in X$, fix one path of length $\ell$ containing
$e_f$ from $v$ to $v_\ell$. Let $I$ and $J$ be disjoint subsets of
$[k']$, and assume that for every $\eta \notin I$ there exists an
orientation $o_\eta$ such that for every $v \in X$, there exists a
path from $v$ to $v_\ell$ that has no broken edge in $P_\eta$,
and traverses $P_\eta$ in direction $o_\eta$. Then
\[ |N(X) \cap P_J| \ge |S_P \cap P_J| - |P_I| - 3\varepsilon^4 n. \]
\end{prop}
\begin{proof}
For every $\eta \notin I$, since all the paths traverse $P_\eta$ in
direction $o_\eta$, we know that when a vertex in $N(X) \cap P_\eta$
is used as a pivot point, it will create a broken edge in a fixed
direction (to the left of the pivot point if $o_\eta$ is positive,
and to the right of the pivot point otherwise). Therefore we have
\[ |N(X) \cap P_{[k'] \setminus (I \cup J)}| + |S_P \cap P_I| + |S_P \cap P_J| \le |S_P|. \]
Consequently,
\begin{align*}
 |N(X)| &= |N(X) \cap \overline{P'}| + |N(X) \cap P_I| + |N(X) \cap P_J| + |N(X) \cap P_{[k'] \setminus (I \cup J)}| \\
 &\le |V(P) \setminus P'| + |P_I| + |N(X) \cap P_J| + (|S_P| - |S_P \cap P_I| - |S_P \cap P_J|),
\end{align*}
and by our assumption that $|S_P| \leq \big( \frac{1}{2} + \varepsilon^4 \big) n$, we have
\begin{align*}
 |N(X)| \le \left( \frac{1}{2} + 2\varepsilon^4 \right) n + |P_I| + |N(X) \cap P_J| - |S_P \cap P_J|.
\end{align*}
On the other hand, since $|X| \ge \frac{n}{\varepsilon^4 r}$, we have $|N(X)| \ge (\frac{1}{2} - \varepsilon^4)n$.
By combining the two bounds we get
$|S_P \cap P_J| - |P_I| - 3\varepsilon^4 n \le |N(X) \cap P_J|$.
\end{proof}

Let $a_{1}$ be the smallest positive integer such that for
$k_{1}=k'-\varepsilon k \cdot a_{1}$, there exist at least
$2\varepsilon^{2}n$ elements of $S_P$ in
$Q_{1}=P_{[k_1+1,k_1+\varepsilon k]}$ (note that
$|Q_1| \le \varepsilon n$). By construction, there exist at most
$2\varepsilon n$ elements of $S_P$ in $P_{[k_{1}+\varepsilon k+1,k']}$, 
and at least $|S_P \cap P'|-2\varepsilon
n-|Q_{1}|\ge(\frac{1}{2}-5\varepsilon)n$ vertices in
$P_{[1,k_{1}]}$. Let $a_{2}$ be the smallest
positive integer such that for $k_{2}=k_{1}-\varepsilon^{2}k \cdot
a_{2}$, there exist at least $\varepsilon^{3}n$ points of $S_P$ in
$Q_{2}=P_{[k_{2}+1,k_{2}+\varepsilon^{2}k]}$ (note
that $|Q_2| \le \varepsilon^2 n$). Note that there exist at most
$\varepsilon n$ vertices of $S_P$ in $P_{[k_2 + \varepsilon^2 k + 1,
k_1]}$. Let $Q_{3}=P_{[1,k_2]}$.

We defined the sets $Q_1$ and $Q_2$ so that the numbers of vertices
of $S_P$ in both of these sets are quite large. Our goal now is to
find a large number of vertices in $Q_1 \cap S_P$ that can be
obtained by two different rotations, one rotation giving a path that
traverses $Q_3$ positively and the other giving a path that traverses $Q_3$ negatively. To do this,
first, we will perform two rotations, where we begin by finding
endpoints in $Q_2$ using $S_t$, and then use these endpoints to find endpoints
in $Q_1$. Second, we will directly rotate from $S_t$ 
to obtain endpoints in $Q_1$. Since both ways will give a big
proportion of vertices in $Q_1 \cap S_P$, we eventually will find
the set of vertices that we wanted.

We will use Proposition \ref{prop:keyprop} to formalize this idea.
Recall that for $1\le i\le k$, each path from $v\in S_{t}$ to
$v_{\ell}$ traverses the interval $P_{i}$ positively. Let $Y_2 =
N(S_t)^- \cap Q_2$. By our construction we have that for $J = [k_2 + 1, k_2 + \varepsilon^2 k]$,
$|S_P \cap P_J| \geq \varepsilon^3 n$. Thus,
by Proposition \ref{prop:keyprop} with $X =
S_t$, $I = \emptyset$, and $J = [k_2 + 1, k_2 + \varepsilon^2 k]$,
we see that $|Y_2| = |N(S_t) \cap Q_2| \ge (\varepsilon^3 -
3\varepsilon^4) n$. Since we can use the points in $N(S_t) \cap Q_2$
as pivot points to get new endpoints $Y_2$, we have that for every
$v \in Y_2$, there exists a path of length $\ell$ containing $e_f$
which starts at $v$ and ends at $v_\ell$. Moreover, these paths have
exactly one broken edge inside $P'$, and it is in $Q_2$. Thus we can
apply Proposition \ref{prop:keyprop} again with $X = Y_2$, $I = [k_2
+ 1, k_2 + \varepsilon^2 k]$, and $J = [k_1 + 1, k_1 + \varepsilon
k]$ to get $|N(Y_2)^- \cap Q_1| \ge |S_P \cap Q_1| - \varepsilon^2 n
- 3\varepsilon^4 n$. For every vertex $v \in N(Y_2)^- \cap Q_1$,
there exists a path of length $\ell$ containing $e_f$ which starts
at $v$ and ends at $v_\ell$. These paths have exactly two broken
edges inside $P'$, both in $Q_1 \cup Q_2$. Moreover, for every
interval in $Q_3$, all these paths traverse the interval in the
positive direction. Now apply Proposition \ref{prop:keyprop} with $X
= S_t$, $I = \emptyset$, and $J = [k_1 + 1, k_1 + \varepsilon k]$ to
get $|N(S_t)^- \cap Q_1| \ge |S_P \cap Q_1| - 3\varepsilon^4 n$. The
vertices in $N(S_t)^- \cap Q_1$ have similar properties to those in
$N(Y_2)^- \cap Q_1$, but the paths for the vertices $N(S_t)^- \cap
Q_1$ traverse the intervals in $Q_3$ in the negative direction (see
Figure \ref{fig_twicerotate}).

\begin{figure}[t]
  \centering
    \begin{tikzpicture}

\def\xend{12}

\def\yshiftone{2.2 cm}

   \draw [fill=black] (9.0 cm, \yshiftone) circle (0.5mm);
   \draw [fill=black] (\xend cm , \yshiftone) circle (0.5mm);

   \draw (6.6 cm, \yshiftone) -- (3.3 cm, \yshiftone);
   \draw [-triangle 45] (2.4 cm, \yshiftone) -- (3.3 cm, \yshiftone);
   \draw (2.4 cm, \yshiftone) -- (0 cm, \yshiftone);   
   
   \draw [dotted] (6.6 cm,\yshiftone) -- (7.2 cm,\yshiftone);

   \foreach \x in {7.2,7.8,...,9.0} {
            \draw (\x cm,\yshiftone) -- (\x cm + 0.6cm,\yshiftone);
        }
   \draw [dotted] (9.0 cm,\yshiftone) -- (9.6 cm,\yshiftone);

   \foreach \x in {9.6,10.2,...,11.4} {
            \draw (\x cm,\yshiftone) -- (\x cm + 0.6cm,\yshiftone);
        }

		\draw (0.3 cm,\yshiftone + 0.1cm) -- (0.3 cm,\yshiftone - 0.12 cm) 
							-- (6.2 cm, \yshiftone - 0.12cm) -- (6.2cm,\yshiftone + 0.1cm);
		\draw (3.3 cm, \yshiftone - 0.4cm) node {$Q_3$};	

	\draw (6.3 cm,\yshiftone + 0.1cm) -- (6.3 cm,\yshiftone - 0.12 cm) 
							-- (7.5cm, \yshiftone - 0.12cm) -- (7.5cm,\yshiftone + 0.1cm);
		\draw (6.9cm, \yshiftone - 0.4cm) node {$Q_2$};	

	\draw (8.1 cm,\yshiftone + 0.1cm) -- (8.1 cm,\yshiftone - 0.12 cm) 
							-- (10.5cm, \yshiftone - 0.12cm) -- (10.5cm,\yshiftone + 0.1cm);
		\draw (9.8 cm, \yshiftone - 0.4cm) node {$Q_1$};	

    \draw (0, \yshiftone) .. controls (2.4cm,\yshiftone + 1.1cm)
            and (4.8cm,\yshiftone + 1.1cm) .. (7.2cm, \yshiftone);
    \draw (6.6cm, \yshiftone) .. controls (7.5cm,\yshiftone + 0.5cm)
            and (8.7cm,\yshiftone + 0.5cm) .. (9.6cm, \yshiftone);

    \draw (9.0cm, \yshiftone - 0.4cm) node {$w$};
    \draw (\xend cm, \yshiftone - 0.4cm) node {$v_{\ell}$};

\def\yshifttwo{0 cm}

   \draw [fill=black] (9.0 cm, \yshifttwo) circle (0.5mm);
   \draw [fill=black] (\xend cm , \yshifttwo) circle (0.5mm);

   \draw [dotted] (9.0 cm,\yshifttwo) -- (9.6 cm,\yshifttwo);

   \draw (9.0 cm, \yshifttwo) -- (3.6 cm, \yshifttwo);
   \draw [-triangle 45] (3.6 cm , \yshifttwo) -- (3.0 cm, \yshifttwo);
   \draw (3.0 cm, \yshifttwo) -- (0 cm, \yshifttwo);

   \foreach \x in {9.6,10.2,...,11.4} {
            \draw (\x cm,\yshifttwo) -- (\x cm + 0.6cm,\yshifttwo);
        }

		\draw (0.3 cm,\yshifttwo + 0.1cm) -- (0.3 cm,\yshifttwo - 0.12 cm) 
							-- (6.2 cm, \yshifttwo - 0.12cm) -- (6.2cm,\yshifttwo + 0.1cm);
		\draw (3.3 cm, \yshifttwo - 0.4cm) node {$Q_3$};	

	\draw (6.3 cm,\yshifttwo + 0.1cm) -- (6.3 cm,\yshifttwo - 0.12 cm) 
							-- (7.5cm, \yshifttwo - 0.12cm) -- (7.5cm,\yshifttwo + 0.1cm);
		\draw (6.9cm, \yshifttwo - 0.4cm) node {$Q_2$};	

	\draw (8.1 cm,\yshifttwo + 0.1cm) -- (8.1 cm,\yshifttwo - 0.12 cm) 
							-- (10.5cm, \yshifttwo - 0.12cm) -- (10.5cm,\yshifttwo + 0.1cm);
		\draw (9.8 cm, \yshifttwo - 0.4cm) node {$Q_1$};

    \draw (0, \yshifttwo) .. controls (3.5cm,\yshifttwo + 1.1cm)
            and (6.1cm,\yshifttwo + 1.1cm) .. (9.6cm, \yshifttwo);
            
    \draw (9.0cm, \yshifttwo - 0.4cm) node {$w$};
    \draw (\xend cm, \yshifttwo - 0.4cm) node {$v_{\ell}$};

\end{tikzpicture}
  \caption{Rotating in two different ways to get the same endpoint.}
  \label{fig_twicerotate}
\end{figure}
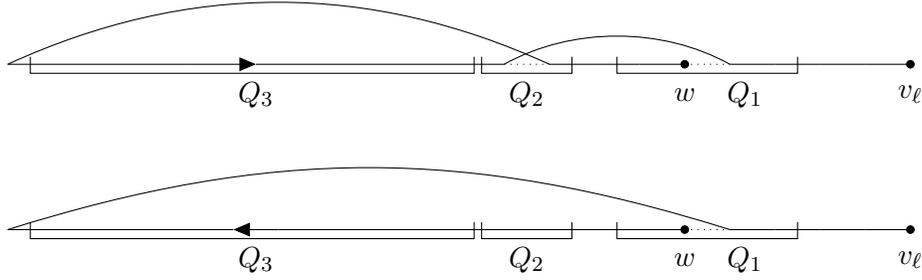

Since both $N(Y_2)^- \cap Q_1$ and $N(S_t)^- \cap Q_1$ are
subsets of $S_P$ and since we constructed $Q_1$ such that $|S_P \cap Q_1| \geq 
2\varepsilon^2 n$, we have for $Y_1 = N(S_t)^- \cap N(Y_2)^- \cap Q_1$,
\[ |Y_1| \ge |S_P \cap Q_1| - \varepsilon^2 n - 6\varepsilon^4 n \ge (\varepsilon^2 - 6\varepsilon^4) n. \]
Note that by using Proposition \ref{prop:keyprop} with
$X= Y_1$, $I = [k_1 + 1, k_1 + \varepsilon k] \cup [k_2 +1, k_2 + \varepsilon^2 k]$, $J = [1, k_2]$, we get
\[ |N(Y_1) \cap Q_3|  \ge |S_P \cap Q_3| - |Q_1| - |Q_2| - 3\varepsilon^4 n \ge |S_P \cap Q_3| - 5\varepsilon n. \]
By the observations made in the previous paragraph, we see that for
every $v \in Y_1$, there exist two paths of length $\ell$ starting
at $v$ and ending at $v_\ell$ such that both paths have at most two
broken edges in $P'$, all belonging to $Q_1 \cup Q_2$. Furthermore,
one of the paths traverses the intervals in $Q_3$ positively, and the
other negatively. Therefore both $N(Y_1)^- \cap Q_3$ and $N(Y_1)^+
\cap Q_3$ are subsets of $S_P$.

The above implies that $|(N(Y_{1})^{-}\cap Q_{3}) \Delta (S_P \cap Q_3)| \le 5\varepsilon n$.
Thus $|(N(Y_1) \cap Q_3) \Delta (S_P^+ \cap Q_3)| \le 5 \varepsilon n$. Similarly we have
$|(N(Y_1) \cap Q_3) \Delta (S_P^- \cap Q_3)| \le 5 \varepsilon n$. By the triangle
inequality we get
\[ |(S_P^- \cap Q_3) \Delta (S_P^+ \cap Q_3)| \le 10 \varepsilon n. \]
Since, by our construction, there are at most $2 \varepsilon n$ vertices of $S_P$ in $P_{[k_{1}+\varepsilon k+1,k']}$ and at most
$\varepsilon n$ vertices of $S_P$ in $P_{[k_2 + \varepsilon^2 k + 1,k_1]}$, 
we can conclude that $|S_P\cap\overline{Q_{3}}|\le |V(P) \setminus P'| + |Q_1| + 
|Q_2| + 3\varepsilon n \le 6\varepsilon n$. This implies that
\[ |S_{P}^{-}\Delta S_{P}^{+}|\le |(S_P^- \cap Q_3)\Delta(S_P^+ \cap Q_3)| + 2|S_P \cap \overline{Q_3}| \le 22\varepsilon n, \]
and completes the proof.
\end{proof}

\subsection{Rotation-extension for bipartite graphs}

In this subsection, we develop a framework useful in tackling the
third case of Lemma \ref{lem:dirac}. Note that the given graph in
this case has a partition of its vertex set so that there are only
few non-adjacent pairs between the two parts. The following
definition gives a nice structure that can be used in this kind of
graphs.

\begin{defn} \label{def:almostbipartite}
Let $G$ be a graph over a vertex set $V$, and let $V_1 \cup V_2$ be
a partition of $V$ satisfying $|V_1| = |V_2| + k$ for some
non-negative integer $k$.
\begin{enumerate}[(i)]
  \setlength{\itemsep}{1pt} \setlength{\parskip}{0pt}
  \setlength{\parsep}{0pt}
\item A tuple $(V_1, V_2, S_V, S_E)$ is
a {\em special frame} of $G$ if $S_V$ is a subset of $V_1$ of
size $k$, and $S_E$ is a set of $k$ vertex-disjoint edges of $G$
in $V_1$ such that each vertex in $S_V$ is incident to exactly
one edge in $S_E$. We refer to $S_V$ as the set of {\em special
vertices}, and $S_E$ is the set of {\em special edges}. Let the
{\em framed subgraph} of $G$ be the subgraph induced by the
edges between $(V_1, V_2)$. Let $V_1' = V_1 \setminus S_V$, and
$V_1'' = V_1 \setminus V(S_E)$.
\item
A tuple $(V_1, V_2, S_V, S_E, f)$ is a {\em matched special frame}
of $G$ if $(V_1, V_2, S_V, S_E)$ is a special frame of $G$, and $f$ is a
perfect matching between the vertices of $V_1'$ and $V_2$.
For $v \in V_1' \cup V_2$, we let $f(v)$ be the vertex matched to $v$ in this matching.
\end{enumerate}
\end{defn}

Throughout this subsection, we fix a vertex set $V$ on $n$ vertices
and a partition $V_1 \cup V_2$ of it satisfying $|V_1| = |V_2| + k$
for some non-negative integer $k$. We then assume that all the
graphs appearing in this subsection are defined over $V$ (therefore
if there are several graphs under consideration, then they share the
same vertex set and its partition). Suppose that we are given a
graph with some special frame $(V_1, V_2, S_V, S_E)$. If $k=0$, then
it suffices to use the edges between the two parts to find a
Hamilton cycle. However, if $k > 0$, then we must use some edges
within $V_1$, and the special edges will be these edges.

To construct a Hamilton cycle in graphs with a matched special
frame, it is easier to consider only a certain class of paths and
cycles.

\begin{defn}
Let $(V_1, V_2, S_V, S_E, f)$ be a matched special frame of some
graph. We say that a path $P$ is a \emph{proper path with respect to
the frame} if it satisfies the following properties: (i) $V_1' \cap
V(P) = f(V_2 \cap V(P))$, (ii) if $P$ contains a special vertex,
then it also contains the special edge incident to it, and (iii) $P$
consists only of edges that intersect both $V_1$ and $V_2$ and of
special edges. We also say that a cycle $C$ is a \emph{proper cycle
with respect to the frame} if it satisfies properties (i),(ii), and
(iii) above (with $C$ replacing $P$). We simply say that a path or
cycle is \emph{proper} if the frame is clear from the context.
\end{defn}

Note that a proper path always has one of its endpoints in $V_1$,
and the other in $V_2$. Indeed, suppose that $P$ is a proper path
with $s$ special edges. Then by properties (i) and (iii), $P$ has
length $2|V_2 \cap V(P)| + s - 1$, and thus switches between $V_1$
and $V_2$ in total $2|V_2 \cap V(P)| - 1$ times. Since this is an
odd number, we can see that the above holds.

We now specify the roles of the graphs performing rotations and
extensions.

\begin{defn}
\label{def:bip-rotation} Let $\xi$ be a positive constant. A graph
$G$ has property $\BIPRE(\xi)$ if it contains a matched special
frame whose framed subgraph is connected, and for every proper path
$P$, at least one of the following holds: (i) there exists a proper
path longer than $P$ in $G \cup P$, or (ii) there exists a set of
vertices $S_{P} \subset V_2$ of size at least $|S_{P}|\ge \xi n$
such that for every vertex $v\in S_{P}$, there exists a set $T_{v}
\subset V_1$ of size $|T_{v}|\ge \xi n$ such that for every $w\in
T_{v}$, there exists a proper path of the same length as $P$ in
$G\cup P$ that starts at $v$ and ends at $w$.
\end{defn}

\begin{defn}
\label{def:bip-extension}Let $\xi$ be a positive constant,
and let $G_{1}$ be a graph with property $\BIPRE(\xi)$. We say that a
graph $G_{2}$ \emph{complements} $G_{1}$,
if for every proper path $P$, (i) there exists
a proper path longer than $P$ in $G_{1} \cup P$, or (ii)
there exist vertices $v\in S_{P}$ and $w\in T_{v}$ such
that $\{v,w\}$ is an edge of $G_{1}\cup G_{2} \cup P$ (the sets $S_{P}$
and $T_{v}$ are as defined in Definition \ref{def:bip-rotation}).
\end{defn}

Two graphs as in the above two definitions together give
Hamiltonicity.

\begin{prop} \label{prop:bip-hamiltonian}
Let $\xi$ be a positive constant.
If $G_{1}\in\BIPRE(\xi)$, and $G_{2}$ complements
$G_{1}$, then $G_{1}\cup G_{2}$ is Hamiltonian.
\end{prop}
\begin{proof}
Let $f$ be the matching in the matched special frame of $G_1$ as in
Definition \ref{def:almostbipartite}, and $P$ be a longest proper
path in $G_{1}\cup G_{2}$. Note that we can use the frame of $G_1$
also as a frame of $G_1 \cup G_2$. Since $G_2$ complements $G_1$, by
Definition \ref{def:bip-extension}, we can find a proper cycle
$(v_{0},\cdots,v_{\ell},v_{0})$ in $G_1 \cup G_2$ over the vertex
set of $P$. Assume that this cycle is not Hamiltonian. Then by the
connectivity of the framed subgraph of $G_1$, there exists an edge
of the form $\{x, v_{i}\}$ for some vertex $x$ not in the cycle,
where $x$ and $v_i$ belong to different parts of the frame. We claim
that this violates the maximality assumption on $P$. This will imply
that the graph $G_1 \cup G_2$ is Hamiltonian.

Indeed, either the vertex $x$ is a special vertex or not. If $x$ is
a special vertex, then since $\{x,v_i\}$ cannot be a special edge,
we must have a special edge $\{x,x'\}$ for some $x' \neq v_i$. If
$x'$ is on the path $P$, then we immediately obtain a proper path
longer than $P$. Thus we may assume that $x'$ is not on the path. In
this case we can use the path $(v_i, x, x', f(x'))$ to find a longer
proper path (note that $f(x')$ is also not in the path $P$ by the
definition of a proper path). Finally, if $x$ is not a special
vertex, then we can use the path $(v_i, x, f(x))$ to find a longer
proper path (again, $f(x)$ is not in the path $P$ by the definition
of a proper path). Thus in any case, we deduce a contradiction.
\end{proof}

As in the previous subsection, we provide a list of deterministic properties
which when satisfied, imply property $\BIPRE(\xi)$.
\begin{defn}
\label{def:bip-expander}
Let $\varepsilon$, $r$ be positive constants, and let $k$ be a non-negative
integer. We say that a graph $G$ is a
{\em $k$-bipartite-expander with parameters $\varepsilon$ and $r$},
if it contains a special frame $(V_1, V_2, S_V, S_E)$ with
$|V_1| = |V_2|+k$ such that the following properties hold:
\begin{enumerate}[(i)]
  \setlength{\itemsep}{1pt} \setlength{\parskip}{0pt}
  \setlength{\parsep}{0pt}
\item For all $X\subset V_{1}$, if $|X|\le \frac{n}{r^{3/2}}$, then $|N(X)\cap V_{2}|\ge r|X|$,
and if $|X|\ge \frac{n}{r^{3/4}}$, then $|N(X)\cap V_{2}|\ge(1-\varepsilon)|V_{2}|$, and
\item for all $Y\subset V_{2}$, if $|Y|\le \frac{n}{r^{3/2}}$, then $|N(Y)\cap V_{1}''|\ge r|Y|$,
and if $|Y|\ge \frac{n}{r^{3/4}}$, then $|N(Y)\cap V_{1}''|\ge(1-\varepsilon)|V_{1}''|$.
\end{enumerate}
\end{defn}

We often refer to $k$-bipartite-expanders as {\em bipartite-expanders} when $k$ is
clear from the context.

\begin{lem} \label{lem:bipexpander_rotationextension}
For every positive reals $r \ge 16$, $\varepsilon \le \frac{1}{4}$,
and non-negative integer $k \le \frac{n}{30}$, every
$k$-bipartite-expander with parameters $\varepsilon$ and $r$ has
property $\BIPRE(\frac{1}{6})$.
\end{lem}
\begin{proof}
Let $G$ be the given bipartite-expander with parameters
$\varepsilon$ and $r$, and let $(V_1, V_2, S_V, S_E)$ be its special
frame. Throughout the proof, we consider only the edges that belong
to the framed subgraph and the special edges. Given properties (i)
and (ii) of Definition~\ref{def:bip-expander}, it follows that every
connected component of the framed subgraph of $G$ contains at least
$\frac{3}{4}|V_2|$ vertices of $V_2$. Consequently, every two
connected components intersect, and the framed subgraph of $G$ is
connected. We then verify Hall's condition to find a matching $f$
between the vertices of $V_1'$ and $V_2$. By property (i) of
bipartite-expanders, we have $|N(X) \cap V_2| \ge |X|$ for every
subset $X$ of $V_1$ of size at most $\frac{3}{4}|V_2|$. Then by
property (ii), we have $|N(Y) \cap V_1'| \ge |N(Y) \cap V_1''| \ge
|Y|$ for every subset $Y$ of $V_2$ of size at most
$\frac{3}{4}|V_1''|$. Now suppose that Hall's condition is not
satisfied, and that there is some subset $X$ of $V_1'$ of size larger than
$\frac{3}{4}|V_2|$ that satisfies $|X| > |N(X)| \ge
\frac{3}{4}|V_2|$ (the second inequality follows from the fact we
established above for subsets of size at most $\frac{3}{4}|V_2|$).
Then it implies that all the neighbors  of $V_2 \setminus
N(X)$  in $V_1'$ belong to $V_1' \setminus X$. However, since
\[ |V_2 \setminus N(X)| = |V_2| - |N(X)| \le \frac{1}{4}|V_2| < \frac{3}{4}|V_1''|, \]
(we used the fact $k \le \frac{n}{30}$ in the last inequality) we
thus must have $|V_2 \setminus N(X)| \le |N(V_2 \setminus N(X))| \le
|V_1' \setminus X|$ which
is a contradiction since $|V_1' \setminus X| = |V_1'| - |X|$ and
$|V_2 \setminus N(X)| = |V_2| - |N(X)| > |V_2| - |X|$. Thus Hall's
condition holds, and we may let $f$ be one fixed perfect matching
between the vertices of $V_1'$ and $V_2$. Consider the matched
special frame $(V_1, V_2, S_V, S_E, f)$.

Let us now focus on verifying the remaining condition given in
Definition~\ref{def:bip-rotation}. Let $P=(v_0,\cdots,v_\ell)$ be a
proper path of length $\ell$ where $v_0\in V_{2}$ (recall that one
endpoint of every proper path is in $V_2$). For a subset $X$ of
vertices of $P$, we use notations $X^-$ and $X^+$ as in the proof of
Lemma \ref{lem_halfexpander}. We will first construct the set $S_P$.
This will be done by iteratively constructing sets $S_{i}\subset
V_{2}$ for $i\ge0$ where $S_i$ has the property that for every $v\in
S_{i}$, there exists a proper path of length $\ell$ which starts at
$v$ and ends at $v_\ell$. Moreover, $S_{i} \subset S_{i+1}$ will
hold for every $i \ge 0$.

Let $S_{0}=\{v_0\}$. Assume that we have completed constructing
$S_{i}$ for some $i\ge0$. We first claim that $N(S_{i})\cap V_{1}
\subset V(P)$. If the claim does not hold, then we have an edge
$\{x,y\}$ for $x\in S_{i},\: y\in V_1\setminus V(P)$. If $y$ is not
a special vertex, then we can find a proper path longer than $P$ by
attaching the path $(f(y),y,x)$ to the proper path of length $\ell$
starting at $x$ and ending at $v_\ell$ (note that $f(y)$ is not in
$P$ by the definition of a proper path). On the other hand, if $y$
is a special vertex, then let $\{y,y'\}$ be the special edge
incident to $y$. If $y' \in V(P)$, then the edges $\{x,y\}$ and $\{y,y'\}$, together
with $P$ immediately give a proper path longer than $P$. Finally,
if $y' \notin V(P)$, then we can use the path $(x,y,y',f(y'))$ to
find a longer proper path than $P$ ($f(y')$ is not in $P$ by the
definition of a proper path). Thus we indeed must have
$N(S_i) \cap V_1 \subset V(P)$.

Now, if a vertex $w \in N(S_i) \cap V_1''$ is not in $S_{i}^-$ or
$S_{i}^+$, then $w$ can be used as a pivot point to give either
$w^-$ or $w^+$ as a new endpoint of some proper path that has
$v_\ell$ as its other endpoint (recall that $S_i$ contains the sets
$S_j$ for $j \le i$). Note that it is crucial to consider the set
$V_1''$ as otherwise we might end up breaking some special edges of
the path $P$ and the resulting path will no longer be proper.
Let $S_{i+1}$ be the union of the set of endpoints
obtained in this way and the set $S_i$. Since an endpoint can be
obtained in at most two ways, we have \[
|S_{i+1}|\ge\frac{1}{2}\Big(|N(S_{i}) \cap
V''_1|-|S_{i}^{+}|-|S_{i}^{-}|\Big) + |S_i| \ge \frac{1}{2}|N(S_{i})
\cap V_1''|.
\] If $|S_{i}|\le \frac{n}{r^{3/2}}$, then we have $|N(S_{i})\cap V_1''|\ge
r|S_i|$ and therefore, $|S_{i+1}|\ge\frac{r}{2}\cdot|S_{i}|$. Since
$r \ge 16$, at some point $t$, $S_{t}$ will have size $|S_{t}|\ge
\frac{n}{r^{3/2}}$. Redefine $S_{t}$ as an arbitrary subset of it of
size $\frac{n}{r^{3/2}}$ and repeat the above once more to get
$|S_{t+1}|\ge\frac{r}{2}|S_{t}|\ge \frac{n}{r^{3/4}}$ (recall that $r \ge 16$).
Again redefine $S_{t+1}$ as an arbitrary subset of it of size
$\frac{n}{r^{3/4}}$, and repeat the above for the final time, to get
a set $S_{P}=S_{t+2}$ of size
\[ |S_{P}|\ge \frac{1}{2} \cdot (1-\varepsilon)|V_1''| \ge
\frac{3}{8}\left(\frac{n}{2} - \frac{3k}{2}\right) \ge \frac{n}{6},
\] where we used the fact that $\varepsilon \le \frac{1}{4}$ and $k \le \frac{n}{30}$.
For each vertex $v\in S_{P}$, we can perform the same process as
above to the other endpoint $v_\ell$ to find a set $T_{v}$ of size
at least $\frac{n}{6}$ satisfying the property of Definition
\ref{def:bip-rotation}.
\end{proof}

In fact, similar arguments also apply to non-bipartite graphs.

\begin{defn}
\label{def:realexpander} For positive constants $\varepsilon$ and
$r$, we say that a graph $G$ on $n$ vertices
is an {\em expander with parameters
$\varepsilon$ and $r$}, if the following properties hold:
\begin{enumerate}[(i)]
  \setlength{\itemsep}{1pt} \setlength{\parskip}{0pt}
  \setlength{\parsep}{0pt}
\item For every subset of vertices $X$ of size $|X|\le \frac{n}{r^{3/2}}$, we have $|N(X)|\ge
r|X|$, and
\item for every subset of vertices $X$ of size $|X|\ge \frac{n}{r^{3/4}}$, we have $|N(X)|\ge
(1-\varepsilon)n$.
\end{enumerate}
\end{defn}

The proof of Theorem \ref{lem:bipexpander_rotationextension} given
above can be easily modified to give the following result.

\begin{lem} \label{lem:expander_rotationextension}
For every positive reals $r \ge 16$ and $\varepsilon \le
\frac{1}{4}$, every expander with parameters $\varepsilon$ and $r$
has property $\mathcal{RE}(\frac{1}{6})$.
\end{lem}

\section{Random subgraphs of Dirac graphs}

In this section we prove Theorem~\ref{thm_mainrandom}.
The following well-known concentration result (see, for example
\cite{AlSp}, Appendix A) will be used several times throughout the
proof. We denote by $Bi(n,p)$ a binomial random variable with
parameters $n$ and $p$.
\begin{thm}[Chernoff's Inequality] \label{thm_chernoffinequality}
If $X \sim Bi(n,p)$ and $\lambda \leq np$, then
\[ P \big( |X - np| \geq \lambda \big) \leq e^{-\Omega(\lambda^2/(np))}. \]
\end{thm}

We begin by applying Lemma~\ref{lem:dirac} with $\alpha =
\frac{1}{2^{23}}$ and $\gamma = \frac{1}{32}$ to classify Dirac
graphs into three types. We will show case by case that a random
subgraph of a Dirac graph of each type is a.a.s.~Hamiltonian.

\subsection{First case}

Let $G$ be a Dirac graph satisfying $e(A,B)\ge\alpha n^{2}$ for all half-sets $A$ and $B$.

\begin{lem} \label{lem_case1_lem1}
There exist positive reals $\varepsilon$
and $\beta$ such that for every $p\ge \frac{\beta \log n}{n}$
the graph $G_{p}$ a.a.s. contains a subgraph with property
$\mathcal{RE}(\frac{1}{2}+\varepsilon)$ that has at most $\beta n \log n$ edges.
\end{lem}
\begin{proof}
Let $\varepsilon = \min\{\varepsilon_{\ref{lem_halfexpander}},
(\frac{\alpha}{4})^{5}, \frac{1}{2^4e^2}\}$, and $\beta \ge 512\varepsilon^{-3}$ be a large
constant. 
Let $p' = \frac{\beta \log n}{n}$ and suppose that we are
given $p \ge p'$. Let $H$ be a random subgraph of $G_p$ obtained by
taking each edge independently with probability $\frac{p'}{p}$, and
note that the distribution of $H$ is identical to that of $G_{p'}$.
Since $G_{p'}$ a.a.s. has at most $(1+o(1))e(G) p' \le \beta n \log
n$ edges, it suffices to show that $G_{p'}$ a.a.s.~has property
$\mathcal{RE}(\frac{1}{2} + \varepsilon)$. By Lemma
\ref{lem_halfexpander}, we can prove our claim by verifying that
$G_{p'}$ is a.a.s.~a half-expander with parameters
$\varepsilon$ and $r = \frac{\beta}{32}\log n = \frac{np'}{32}$.

We will establish the following four properties of $G_{p'}$
which together will imply that $G_{p'}$
is a.a.s.~a half-expander with parameters $\varepsilon$ and $r$
(note that $n p' = \beta \log n$).

\begin{enumerate}
  \setlength{\itemsep}{1pt} \setlength{\parskip}{0pt}
  \setlength{\parsep}{0pt}
\item $G_{p'}$ a.a.s.~has minimum degree at least $\frac{1}{3}np'$.
\item For every pair of sets $X$ and $Y$ satisfying $|X| \leq  \frac{\varepsilon n}{r}$ and $|Y| \le r|X|$, we
a.a.s.~have $e_{G_{p'}}(X,Y) \le \frac{1}{4}|X|np'$.
\item For every pair of sets $X$ and $Y$ of size $|X| \ge \frac{n}{\varepsilon r}$ and $|Y| \ge (\frac{1}{2}+\varepsilon)n$,
$G_{p'}$ a.a.s.~contains at least one edge between $X$ and $Y$.
\item For every pair of sets $X$ and $Y$ of size $|X|, |Y| \ge (\frac{1}{2} - \varepsilon^{1/5})n$,
$G_{p'}$ a.a.s.~contains at least $2n+1$ edges between $X$ and $Y$.
\end{enumerate}
Indeed, suppose that all four properties hold. Let $X$ be a set of
size at most $\frac{\varepsilon n}{r}$ and assume that $|N(X)| \le r|X|$.
For the set $Y = N(X)$, by Properties 1 and 2, we must
have
\[ \frac{1}{4} |X|np' \ge e_{G_{p'}}(X, Y) = e_{G_{p'}}(X, V(G))\ge |X| \cdot \frac{1}{3}np' \]
which is a contradiction. Thus Condition (i) of Definition
\ref{def:expander} holds. Now let $Z$ be a set of size at least
$\frac{n}{\varepsilon r}$ and assume that $|N(Z)| \le (\frac{1}{2} -
\varepsilon)n$. If this is the case, then there are no edges between
$Z$ and $V \setminus N(Z)$ in $G_{p'}$, and this contradicts Property 3 since $|V
\setminus N(Z)| \ge (\frac{1}{2} +\varepsilon)n$. This establishes
Condition (ii) of Definition \ref{def:expander}. Finally, it is easy
to see that Property 4 implies Condition (iii) of Definition
\ref{def:expander}.


Now we establish the four properties listed above. Since $G$ has
minimum degree at least $\frac{n}{2}$, by Chernoff's inequality, the
probability of a fixed vertex having degree less
$\frac{1}{3}np'$ is at most $e^{-\Omega(np')}= o(n^{-1})$ for large
enough $\beta$. By taking the union bound, we obtain Property 1. For
Property 2, let $X$ and $Y$ be two sets of size $|X|= k$ and
$|Y|=rk$ (it suffices to prove for $Y$ exactly of size $rk$). To
estimate the probability of $e_{G_{p'}}(X,Y) \ge \frac{knp'}{4}$, we
may estimate the probability that at least $\frac{knp'}{8}$
non-ordered pairs of the form $\{v,w\}$ for $v \in X, w\in Y$,
become edges of $G_{p'}$. Consequently, the probability of having
$e_{G_{p'}}(X,Y) \ge \frac{knp'}{4}$ is at most
\[ {rk^2 \choose knp'/8} (p')^{knp'/8}. \] Therefore by the union
bound, the probability of having two sets violating Property 2 is at
most
\begin{align*}
\sum_{k=1}^{\varepsilon n/r} {n \choose k}{n \choose rk} {rk^2 \choose knp'/8} (p')^{knp'/8}
&\le \sum_{k=1}^{\varepsilon n/r} {n \choose k} \left( \frac{en}{rk} \right)^{rk} \left(\frac{8erk}{n} \right)^{knp'/8} \\
&= \sum_{k=1}^{\varepsilon n/r} {n \choose k} \left( \frac{en}{rk} \big(\frac{8erk}{n}\big)^4 \right)^{rk}
= \sum_{k=1}^{\varepsilon n/r} {n \choose k} \left( \frac{2^{12}e^5r^3k^3}{n^3} \right)^{rk} \\
&\le \sum_{k=1}^{\varepsilon n/r} {n \choose k} \left( \varepsilon^3 2^{12}e^5 \right)^{rk} \\
&\le \sum_{k=1}^{\varepsilon n/r} \left( \frac{en}{k} \Big(\varepsilon^3 2^{12}e^5\Big)^r \right)^{k},
\end{align*}
which is $o(1)$ for large enough $\beta$ since $r = \frac{\beta}{32} \log n=\frac{np'}{32}$,
and $\varepsilon \le 2^{-4}e^{-2}$.

To establish Property 3, let $X$ and $Y$ be two sets as in Property
3. Since the graph $G$ has minimum degree at least $n/2$, the number
of edges in $G$ between $X$ and $Y$ is at least
\[ \frac{1}{2}|X| \cdot \varepsilon n \ge \frac{n^2}{2r}. \]
Consequently the probability of Property 3 not holding is at most
\[ 2^{2n} \cdot (1-p')^{n^2/(2r)} \le 2^{2n} e^{-n^2p'/(2r)} = 2^{2n} e^{-16n} = o(1) \]
(we used the fact that $p' = \frac{\beta \log n}{n}$ and $r =
\frac{\beta}{32}\log n$).

It now remains to verify Property 4. It suffices to consider a pair of
sets $X$ and $Y$ which both have size exactly $(\frac{1}{2} -
\varepsilon^{1/5})n$. Let $X'$ be an arbitrary set of vertices of
size $\frac{n}{2}$ that contains $X$, and similarly define $Y'$.
Then the graph $G$ satisfies $e_G(X',Y') \ge \alpha n^2$.
Therefore we have
\[ e_{G}(X, Y) \ge e_G(X',Y') - |X'\setminus X| \cdot n - |Y'\setminus Y| \cdot n \ge
  \alpha n^2 - 2\varepsilon^{1/5} n^2 \ge \frac{\alpha n^2}{2}, \]
and thus by Chernoff's inequality, the probability that $G_{p'}$ has
less than $\frac{\alpha n^2p'}{4}$ edges between $X$ and $Y$ is at most
$e^{-\Omega(\alpha n^2p')} = e^{-\Omega(\alpha \beta n\log n)}$. Thus by taking the
union bound over all possible choices of $X$ and $Y$, we have
Property 4.
\end{proof}

\begin{lem} \label{lem_case1_lem2}
For every fixed positive reals $\varepsilon$ and $\beta$, there
exists a constant $C=C(\varepsilon, \beta)$ such that the following
holds for every $p\ge \frac{C\log n}{n}$: $G_{p}$ a.a.s. complements
every its subgraph with property
$\mathcal{RE}(\frac{1}{2}+\varepsilon)$ that has at most $\beta n
\log n$ edges.
\end{lem}
\begin{proof}
Let $C \ge \beta$ be a sufficiently large constant. The probability
that the assertion of the lemma fails is
\begin{align}
\mathbf{P} &=& \mathbf{P}\Big(\bigcup_{R\in\mathcal{RE}(\frac{1}{2} + \varepsilon), |E(R)| \leq \beta n \log n}\Big(\{R\subset G_p\} \wedge \{\textrm{$G_p$ does not complement $R$}\}\Big)\Big)  \nonumber \\
&\leq&  \sum_{R\in\mathcal{RE}(\frac{1}{2} + \varepsilon), |E(R)| \leq \beta n \log n}\mathbf{P}\Big(\textrm{$G_p$ does not complement $R$}\,|\, R\subset G_p\Big)\cdot
\mathbf{P}(R\subset G_p) \label{eq:eq1},
\end{align}
where the union (and sum) is taken over all graphs $R$
on $n$ vertices which have property $\mathcal{RE}(\frac{1}{2} + \varepsilon)$
and at most $\beta n \log n$ edges.

Let us first examine the term $\mathbf{P}\Big(\textrm{$G_p$ does not
complement $R$} \,|\, R \subset G_p\Big)$. Let $R$ be a fixed graph
with property $\mathcal{RE}(\frac{1}{2} + \varepsilon)$ and $P$ be a
fixed path on the same vertex set containing some fixed edge $e$.
The number of such paths is at most $n^2 \cdot n!$, since there are
at most $n^2$ choices for the length of path $P$ and the fixed edge $e$, and
there are at most $n(n-1)\cdots(n-i+1)$ paths of length $i, 1 \leq i
\leq n$. If in $R\cup P$ there is a path longer than $P$ containing
$e$, then the condition of Definition \ref{def:extension} is already
satisfied. Therefore we can assume that there is no such path in
$R\cup P$. Then, by the definition of property
$\mathcal{RE}(\frac{1}{2} + \varepsilon)$, we can find a set $S_{P}$
and for every $v \in S_P$ a corresponding set $T_{v}$, both of size
$\big(\frac{1}{2}+\varepsilon\big)n$, such that for every $w\in
T_{v}$, there exists a path containing $e$ of the same length as $P$
in $R\cup P$ which starts at $v$ and ends at $w$. If there exist 
vertices $v\in S_{P}$ and $w\in T_{v}$ such that $\{v,w\}$ is an edge
of $R$, then this edge is also in $R \cup G_p$ and again Definition
\ref{def:extension} is satisfied. If there are no such edges of $R$,
then conditioned on $R\subset G_p$,
each such pair of vertices is an edge in $G_p$ independently with
probability $p$. By the minimum degree condition on $G$, we have
\[ e_G(S_P, T_v) \ge \varepsilon n \cdot \left(\frac{1}{2} + \varepsilon \right)n \ge \frac{\varepsilon n^2}{2}. \]
Therefore there are at least $\frac{1}{2}e_G(S_P, T_v) \ge \frac{\varepsilon n^2}{4}$
edges of $G$ that we would like to be present in the graph
(the factor $\frac{1}{2}$ comes from the fact that a same edge can be counted twice).
If $G_p$ does not complement the graph $R$, then a.a.s.~there
exists some path $P$ such that no such edge appears in $G_p$.
For a fixed path $P$, the probability of this event is at most
$(1-p)^{\varepsilon n^2/4}$. Consequently,
by taking the union bound over all choice of paths $P$, we see that
for large enough $C=C(\varepsilon)$ and $p\ge\frac{C\log n}{n}$
\[
\mathbb{P}\Big(\textrm{$G_p$ does not complement}\: R\,|\, R\subset G_p\Big)\le n^2\cdot n!\cdot (1-p)^{-\varepsilon n^2 /4}
\le e^{-\varepsilon n^2 p/8}.
\]
Therefore in \eqref{eq:eq1}, the right hand side can be bounded by
\[ \mathbf{P} \leq e^{-\varepsilon n^2 p/8}\cdot\sum_{R\in\mathcal{RE}(\frac{1}{2} + \varepsilon), |E(R)| \leq \beta n \log n}\mathbf{P}(R\subset G_p).\]
Also note that for a fixed graph $R$ with $t$ edges
$\mathbf{P}(R \subset G_p) \le \mathbf{P}( R \subset G(n,p))=p^t$. Therefore,
by taking the sum over all possible graphs $R$ with at most $\beta n \log n$
edges, we can bound the probability that the assertion of the lemma fails by
\[ \mathbf{P} \leq  e^{-\varepsilon n^2 p/8} \sum_{t=1}^{\beta n \log n}{{n \choose 2} \choose t}p^{t}
\le e^{-\varepsilon n^2 p/8} \sum_{t=1}^{\beta n \log
n}\Big(\frac{en^{2}p}{t}\Big)^{t}.\] Since $p \ge \frac{C\log
n}{n}$, for $C \ge \beta$, the summands are monotone increasing in
the range $1 \le t \le \beta n \log n$, and thus we can take the
case $t=\beta n \log n$ for an upper bound on every term. This gives
\[
\mathbf{P} \le (1+o(1)) \beta n \log n \cdot
   \left( e^{-\varepsilon n p/ (8\beta \log n)} \cdot \Big(\frac{enp}{\beta \log n}\Big) \right) ^{\beta n \log n},
\]
which is $o(1)$ for sufficiently large $C$ depending on
$\varepsilon$ and $\beta$, since $p \ge \frac{C\log n}{n}$. This
completes the proof.
\end{proof}

\begin{thm}
There exists a constant $C$ such that the
following holds for every $p\ge \frac{C\log n}{n}$.
If $G$ is a Dirac graph satisfying $(i)$ of Lemma \ref{lem:dirac},
then $G_{p}$ is a.a.s. Hamiltonian.
\end{thm}
\begin{proof}
Let $\varepsilon = \varepsilon_{\ref{lem_case1_lem1}}$,
$\beta = \beta_{\ref{lem_case1_lem1}}$, and
$C = \max\{\beta, C_{\ref{lem_case1_lem2}}(\varepsilon, \beta) \}$.
By Lemma \ref{lem_case1_lem1}, we know that $G_p$ a.a.s.~contains a subgraph
that has property $\mathcal{RE}(\frac{1}{2} + \varepsilon)$ and at most $\beta n \log n$ edges.
Then by Lemma \ref{lem_case1_lem2}, we know that $G_p$ a.a.s.~complements this
subgraph. Therefore when both events hold, we see by Proposition \ref{prop:rotationextension} that
$G_p$ is a.a.s. Hamiltonian.
\end{proof}

\subsection{Second case}

Let $G$ be a Dirac graph satisfying the following as in $(ii)$ of Lemma \ref{lem:dirac}:
there exists a set $A$ of size $\frac{n}{2} \le|A|\le (\frac{1}{2} + 16\alpha)n$ such
that $e(A,\overline{A})\le 6\alpha n^{2}$, and the induced subgraphs
on both $A$ and $\overline{A}$ have minimum degree at least $\frac{n}{5}$.
Let $k=|A|-|\overline{A}|$ so that $|A| = \frac{n+k}{2}$, $|\overline{A}| =
\frac{n-k}{2}$, and $k \le 32\alpha n$.

Note that
\[ |A|\frac{n}{2} \le e(A, V(G)) = e(A,A) + e(A, \overline{A}) \le 2e(A) + 6\alpha n^2. \]
Therefore,
\begin{align}
e(A) \ge& \frac{1}{2}\left( \frac{|A|n}{2} - 6\alpha n^2 \right)
           =  \frac{1}{2}\left( |A|^2   -  \left(|A| - \frac{n}{2}\right)|A| - 6\alpha n^2 \right) \nonumber \\
       \ge& {|A| \choose 2} - \frac{k|A|}{4} - 3\alpha n^2 \ge {|A| \choose 2} - 11\alpha n^2.
       \label{eq:lem_case2}
\end{align}
Similarly, we can show that $e(\overline{A}) \ge {|\overline{A}| \choose 2} - 6\alpha n^2$.

\begin{lem} \label{lem_case2_lem1}
There exists a positive real $\beta$ such that
the following holds for every $p\ge \frac{\beta \log n}{n}$:
each of the graphs $G[A]_{p}$ and $G[\overline{A}]_{p}$ a.a.s. contains a subgraph
with property $\mathcal{RE}(\frac{1}{6})$ that has at most $\beta n \log n$ edges.
\end{lem}
\begin{proof}
We will only verify the property for $G[A]_p$ since the property for
$G[\overline{A}]_p$ can be verified similarly. Let $\beta$ be a
large enough constant. Given $p \ge \frac{\beta \log n}{n}$, let $p'
= \frac{\beta\log n}{n}$ and let $H$ be a random subgraph of
$G[A]_p$ obtained by taking every edge of $G[A]_p$ independently
with probability $\frac{p'}{p}$ (which is less than 1). Then $H$ has
the same distribution as $G[A]_{p'}$. Since $G[A]_{p'}$ a.a.s.~ has
at most $\beta n \log n$ edges, it suffices to show that $H$
a.a.s.~has property $\mathcal{RE}(\frac{1}{6})$. By Lemma
\ref{lem:expander_rotationextension}, it suffices to verify that $H$
is an expander with parameters $\frac{1}{4}$ and $2^{16}$.
Equivalently, we need to verify the following two properties:
\begin{enumerate}[(i)]
  \setlength{\itemsep}{1pt} \setlength{\parskip}{0pt}
  \setlength{\parsep}{0pt}
\item For every subset of vertices $X$ of size $|X|\le \frac{|A|}{2^{24}}$, we have $|N(X)|\ge
2^{16}|X|$, and
\item for every subset of vertices $X$ of size $|X|\ge \frac{|A|}{2^{12}}$, we have $|N(X)|\ge
\frac{3}{4}|A|$.
\end{enumerate}
We claim that $H$ a.a.s.~has the following properties: (a) minimum
degree is at least $\frac{np'}{6}$, (b) for every
pair of sets $X$ and $Y$ of sizes $|X| \le \frac{|A|}{2^{24}}$ and
$|Y| < 2^{16}|X|$, we have $e_H(X,Y) < \frac{|X|np'}{10}$, and
(c) for every pair of sets $X$ and $Y$ of sizes $|X| = \frac{|A|}{2^{12}}$
and $|Y| = \frac{1}{4}|A|$, $e_H(X,Y) > 0$. Suppose that $H$ indeed
satisfies these properties. For every set $X$ of size
$|X| \le \frac{|A|}{2^{24}}$, if $|N(X)| < 2^{16}|X|$, then by
(a) and (b) we will have
\[ \frac{|X|np'}{10} > e_H(X, N(X)) = e_H(X, V(H)) \ge |X|\frac{np'}{6}, \]
which is a contradiction. Thus $(i)$ holds. One can also easily
see that (c) implies $(ii)$.

We omit the proof of Property (a) (we need $\beta$ to be large enough
for (a)) and verify Property (b).
Let $t$ be a positive integer satisfying $t \le \frac{|A|}{2^{24}}$. For a fixed
pair of sets $X$ and $Y$ of sizes $|X| = t$ and $|Y| = 2^{16}t$, in order to have
$e_H(X,Y) \ge \frac{|X|np'}{10}$, at least $\frac{|X|np'}{20}$ non-ordered pairs that have
one endpoint in $X$ and the other in $Y$ must be present. Thus the probability
of the event $e_H(X,Y) \ge \frac{|X|np'}{10}$ is at most
\[ {|X||Y| \choose \frac{|X|np'}{20}} \cdot \left(p'\right)^{\frac{|X|np'}{20}}
 \le \left(\frac{20e|Y|}{n}\right)^{|X|np'/20}. \]
Therefore, by taking the union bound over all choices of $t$ and sets $X$,$Y$,
we see that the probability of $(b)$ being false is
\begin{align*}
& \sum_{t=1}^{|A|/2^{24}} {n \choose t} {n \choose 2^{16}t} \left(\frac{20 \cdot 2^{16} et}{n}\right)^{tnp'/20} \\
\le& \sum_{t=1}^{|A|/2^{24}} \left(\frac{en}{2^{16}t}\right)^{2^{17}t} \left(\frac{20 \cdot 2^{16} et}{n}\right)^{tnp'/20}= \sum_{t=1}^{|A|/2^{24}} \left(\left(\frac{en}{2^{16}t}\right)^{2^{17}} \left(\frac{20 \cdot 2^{16} et}{n}\right)^{np'/20}\right)^{t} \\
=& \sum_{t=1}^{|A|/2^{24}} \left(\left(\frac{e}{2^{16}}\right)^{2^{17}}  \cdot (20\cdot 2^{16})^{2^{17}} \left(\frac{20 \cdot 2^{16} et}{n}\right)^{(np'/20) - 2^{17}}\right)^{t}.
\end{align*}
Since $t \le \frac{n}{2^{24}}$, the summand is maximized at $t=1$, and thus
by $p' = \frac{\beta \log n}{n}$, the probability above can be bounded by
\[ n \cdot (1+o(1))n^{-\Omega(\log n)} = o(1). \]
To verify Property (c), let $X$ and $Y$ be fixed sets of size
$|X| = \frac{|A|}{2^{12}}$ and $|Y| = \frac{1}{4}|A|$. Then
since $e_G(A) \ge {|A| \choose 2} - 11\alpha n^2$ (see \eqref{eq:lem_case2}),
we have
$e_G(X,Y) \ge |X||Y| - 22\alpha n^2 \ge 2^{-15}|A|^2 \ge 2^{-17} n^2$.
Thus by Chernoff's inequality, the probability that $e_H(X,Y) = 0$
is at most $e^{-\Omega(n^2p')} = e^{-\Omega(\beta n \log n)}$.
By taking the union bound over all pairs of sets $X$ and $Y$, we 
obtain (c).
\end{proof}

\begin{lem}  \label{lem_case2_lem2}
For every fixed positive real $\beta$,
there exists a constant $C=C(\beta)$ such that
the following holds for every $p\ge \frac{C\log n}{n}$:
each of the graphs $G[A]_{p}$ and $G[\overline{A}]_{p}$  a.a.s. complements
its every subgraph with property
$\mathcal{RE}(\frac{1}{6})$ that has at most $\beta n \log n$ edges.
\end{lem}
\begin{proof}
It suffices to make a slight modification to the proof of Lemma \ref{lem_case1_lem2}
to prove this lemma. As in Lemma \ref{lem_case2_lem1}, we only consider
the graph $G[A]$. Fix a graph $R$ with property $\mathcal{RE}(\frac{1}{6})$
that has at most $\beta n \log n $ edges, and let
$P$ be a maximum path in $R \cup P$.
Let $S_P$ and $T_v$ be given as in Definition \ref{def:rotation}. Then
since $e_G(A) \ge {|A| \choose 2} - 11\alpha n^2$ (see \eqref{eq:lem_case2}),
we have
\[ e_G(S_P, T_v) \ge |S_P|\cdot |T_v| - 22\alpha n^2 \ge \frac{|A|^2}{36} - 22\alpha n^2 \ge \frac{n^2}{200}.  \]
We can proceed exactly as in Lemma \ref{lem_case1_lem2} to conclude our lemma.
\end{proof}

\begin{thm}
There exists a constant $C$ such that the following holds for every
$p \ge \frac{C\log n}{n}$. If $G$ is a Dirac graph satisfying $(ii)$ of Lemma \ref{lem:dirac},
then $G_{p}$ is a.a.s. Hamiltonian.
\end{thm}
\begin{proof}
Let $\beta = \beta_{\ref{lem_case2_lem1}}$ and $C = \max\{\beta,
C_{\ref{lem_case2_lem2}}(\beta), 4 \}$. By Lemma
\ref{lem_case2_lem1}, we know that a.a.s.~each $G_p[A]$ and
$G_p[\overline{A}]$ contains a subgraph that has property
$\mathcal{RE}(\frac{1}{6})$ and at most $\beta n \log n$ edges. Then
by Lemma \ref{lem_case2_lem2}, we know that $G_p[A]$ and
$G_p[\overline{A}]$ complement each of their subgraphs with the
above property. Therefore when both events hold, we see by
Proposition \ref{prop:rotationextension} that $G_p[A]$ and
$G_p[\overline{A}]$ are Hamiltonian connected. It then suffices to
show that in $G_p$ a.a.s.~there exist two vertex disjoint edges
that connect $A$ and $\overline{A}$, since together with the
Hamilton connectivity of $G_p[A]$ and $G_p[\overline{A}]$ this will
imply Hamiltonicity of $G_p$. In order to prove this, consider the
bipartite graph $\mathcal{B}$ induced by the edges of $G$ between
$A$ and $\overline{A}$, and let $\mathcal{B}_p = \mathcal{B} \cap
G_p$. By Hall's theorem, it suffices to show that $\mathcal{B}_p$
a.a.s.~does not have a vertex that dominates all the edges of
$\mathcal{B}_p$.

Let $v$ be a fixed vertex and first assume that $|A| =
|\overline{A}| = \frac{n}{2}$. Since the minimum degree of $G$ is at
least $\frac{n}{2}$, the graph $\mathcal{B}$ has minimum degree at
least 1. Thus there exist at least $\frac{n}{2} -1$ edges which are
not incident to $v$, and the probability that all the edges of
$\mathcal{B}_p$ are incident to $v$ is at most $(1-p)^{n/2 - 1} \le
e^{- C \log n / 3}$. If $|A| = \frac{n}{2} + t$ for some $t>0$, then
all the vertices of $\overline{A}$ have degree at least $\lceil t  +
1 \rceil \ge 2$ in $\mathcal{B}$. Since $t \le 16\alpha n$, the
total number of edges in $\mathcal{B}$ is at least $2 \cdot
(\frac{n}{2} - t) \ge (1-32\alpha)n$, and
since the maximum degree of $\mathcal{B}$ is at most $|A| \le
(\frac{1}{2} + 16\alpha)n$, the number of edges
not incident to $v$ is at least $(1-32\alpha)n - (\frac{1}{2} + 16\alpha)n \ge
\frac{n}{3}$. Therefore in this case, the probability that all the
edges of $\mathcal{B}_p$ are incident to $v$ is at most $(1-p)^{n/3}
\le e^{- C \log n/3}$.

Thus in either of the cases, for a fixed vertex $v$, the probability
that $v$ dominates all the edges of $\mathcal{B}_p$ is at most
$e^{-C \log n /3}$. Since $C \ge 4$, this probability is
$o(n^{-1})$, and by taking the union bound over all the vertices, we
can conclude that a.a.s.~there is no vertex which is incident to all
the edges of $\mathcal{B}_p$. This concludes the proof.
\end{proof}

\subsection{Third case}

Let $G$ be a Dirac graph satisfying the following as in $(iii)$ of
Lemma \ref{lem:dirac}. There exists a set $A$ of size $\frac{n}{2}
\le|A|\le(\frac{1}{2}+16\alpha)n$ such that the bipartite graph
induced by the edges between $A$ and $\overline{A}$ has at least
$(\frac{1}{4}-14\alpha)n^{2}$ edges and minimum degree at least
$\frac{n}{64}$. Moreover, either $|A|=\lceil \frac{n}{2} \rceil$, or
the induced subgraph $G[A]$ has maximum degree at most
$\frac{n}{32}$. Let $V_1 = A$ and $V_2 = \overline{A}$. Let
$k=|V_1|-|V_2|$ so that $|V_1| = \frac{n+k}{2}$, $|V_2| =
\frac{n-k}{2}$, and $k \le 32\alpha n$.

\begin{lem} \label{lem_case3_lem1}
There exists a constant $\beta$ such that the
following holds for every $p\ge \frac{\beta \log n}{n}$:
$G_{p}$ a.a.s. contains a subgraph with property $\BIPRE(\frac{1}{6})$
that has at most $\beta n \log n$ edges.
\end{lem}
\begin{proof}
Let $\beta \ge 128$ be a large enough constant. Let $p' =
\frac{\beta \log n}{n}$ and note that $p \ge p'$. Let $H$ be a
subgraph of $G_p$ obtained by taking each edge independently with
probability $\frac{p'}{p}$, and note that the distribution of $H$ is
identical to that of $G_{p'}$. Since $G_{p'}$ a.a.s.~has at most
$(1+o(1))p' \cdot  e(G) \le \beta n \log n$ edges, it suffices to
show that $G_{p'}$ a.a.s.~has property $\BIPRE(\frac{1}{6})$.

We first show that $G_{p'}$ contains at least $k$ vertex disjoint
edges in $V_1$. To show this, it suffices to show that $G_{p'}[V_1]$
a.a.s.~has covering number at least $2k-1$. Since this is trivial
for $k =0$, we may assume that $k \ge 1$. By the union bound, we can
bound the probability that the covering number is at most $2k-2$ as
follows:
\begin{equation} \label{eq:coveringnum}
\sum_{X \subset V_1, |X| = 2k-2} \BFP( V_1 \setminus X \textrm{ is an independent set in $G_{p'}$}).
\end{equation}
Note that since $|V_2| = \frac{n-k}{2}$, the induced subgraph
$G[V_1]$ must have minimum degree at least $\frac{k}{2}$. Moreover,
since the graph $G[V_1]$ has maximum degree at most $\frac{n}{32}$,
we can see that for a set $X$ of size $|X| = 2k-2$, the number of
edges of $G$ in $V_1 \setminus X$ is at least
\[ \frac{1}{2} \cdot \frac{k}{2} \cdot |V_1| - (2k-2) \frac{n}{32} \ge \frac{kn}{8} - \frac{kn}{16} = \frac{kn}{16}, \]
and therefore
\[ \BFP( V_1 \setminus X \textrm{ is an independent set}) \le (1-p)^{kn/16} \le e^{-knp/16} \le n^{-\beta k/16}. \]
By using this inequality (note that $\beta \ge 128$), we can bound
\eqref{eq:coveringnum} from above by
\[ {n \choose 2k} \cdot n^{-\beta k/16} = o(1). \]

Consequently we a.a.s.~have $k$ vertex disjoint edges in
$G_{p'}[V_1]$. Condition on this event being true. Arbitrarily pick
$k$ vertex disjoint edges in $V_1$ as our special edges $S_E$, and
for each such edge, declare one of its vertices as a special vertex
(let $S_V$ be the set of special vertices). Note that $G_{p'}$
contains a special frame $(V_1, V_2, S_V, S_E)$. Let $V_1'$, $V_1''$
be as in the Definition \ref{def:almostbipartite}. We will prove
that for large enough $\beta$, conditioned on the special frame,
$G_{p'}$ a.a.s.~has the following two properties:
\begin{enumerate}
  \setlength{\itemsep}{1pt} \setlength{\parskip}{0pt}
  \setlength{\parsep}{0pt}
\item For every $X\subset V_{1}$, if $|X|\le \frac{n}{2^{30}}$, then $|N(X)\cap V_{2}|\ge 2^{20}|X|$,
and if $|X|\ge \frac{n}{2^{15}}$, then $|N(X)\cap V_{2}|\ge
\frac{3}{4}|V_{2}|$, and
\item for every $Y\subset V_{2}$, if $|Y|\le \frac{n}{2^{30}}$, then $|N(Y)\cap V_{1}''|\ge 2^{20}|Y|$,
and if $|Y|\ge \frac{n}{2^{15}}$, then $|N(Y)\cap V_{1}''|\ge
\frac{3}{4}|V_{1}''|$.
\end{enumerate}
This will be done by establishing the following properties that the bipartite subgraph $H$ of
$G_{p'}$ containing all the edges of $G_{p'}$ between $V_1, V_2$ a.a.s.~has: (a) minimum degree is at least
$\frac{np'}{70}$, (b) for every pair of sets $X$ and $Y$ of sizes
$|X| \le \frac{n}{2^{30}}$ and $|Y| < 2^{20}|X|$, we have
$e_H(X,Y) < \frac{|X|np'}{80}$, and (c) for every pair of sets $X$
and $Y$ of sizes $|X| = \frac{n}{2^{15}}$ and $|Y| =
\frac{n}{10}$, $e_H(X,Y) > 0$. To verify (a), we can use the fact
that the bipartite graph induced by the edges between $A$ and
$\overline{A}$ has minimum degree at least $\frac{n}{64}$. The proof of (b) follows from direct application of
Chernoff's inequality and the union bound. To
verify (c), we can use that $e(A, \overline{A}) \ge (\frac{1}{4} -
14\alpha)n^2 \ge |A|\cdot |\overline{A}| - 14\alpha n^2$
and therefore $e(X, Y) \geq |X||Y|-14\alpha n^2 \geq 2\alpha n^2$. Detailed
computation and the derivation of (i), (ii) from (a), (b), and (c)
are similar to that of Lemma \ref{lem_case1_lem1}. Once we verify
these properties we have a bipartite-expander with parameters
$\frac{1}{4}$ and $2^{30}$ (see Definition \ref{def:bip-expander}),
and by Lemma \ref{lem:bipexpander_rotationextension}, we can derive
that our graph has property $\BIPRE(\frac{1}{6})$.
\end{proof}

\begin{lem}  \label{lem_case3_lem2}
For every fixed positive real $\beta$, there exists a constant $C =
C(\beta)$ such that the following holds for every $p\ge \frac{C\log
n}{n}$: $G_{p}$ a.a.s. complements every its subgraph with property
$\BIPRE(\frac{1}{6})$ that has at most $\beta n \log n$ edges.
\end{lem}
\begin{proof}
As in the proof of Lemma \ref{lem_case2_lem2}, it suffices to make a
slight modification to the proof of Lemma \ref{lem_case1_lem2} to
prove this lemma. Fix a graph $R$ with property
$\mathcal{RE}(\frac{1}{6})$ that has at most $\beta n \log n $
edges, and let $P$ be a maximum path in $R \cup P$. Let $S_P$ and
$T_v$ be given as in Definition \ref{def:rotation}. Then since
$e_G(A, \overline{A}) \ge \frac{n^2}{4} - 14\alpha n^2 \ge |A| \cdot
|\overline{A}| - 14\alpha n^2$, we have
\[ e_G(S_P, T_v) \ge |S_P|\cdot |T_v| - 14\alpha n^2 \ge \frac{n^2}{36} - 14\alpha n^2 \ge \frac{n^2}{40}.  \]
We can proceed exactly as in Lemma \ref{lem_case1_lem2} to conclude
our lemma.
\end{proof}

\begin{thm}
There exists a constant $C$ such that the following holds for every
$p \ge \frac{C\log n}{n}$. If $G$ is a Dirac graph satisfying $(iii)$ of Lemma \ref{lem:dirac},
then $G_{p}$ is a.a.s. Hamiltonian.
\end{thm}
\begin{proof}
Let $\beta = \beta_{\ref{lem_case3_lem1}}$, and let $C = \max\{
\beta, C_{\ref{lem_case3_lem2}}(\beta)\}$. By Lemma
\ref{lem_case3_lem1}, we know that $G_p$ a.a.s.~contains a subgraph
that has property $\BIPRE(\frac{1}{6})$ and at most $\beta n \log n$
edges. Then by Lemma \ref{lem_case3_lem2}, we know that $G_p$
a.a.s.~complements this subgraph. Therefore when both events hold,
we see by Proposition \ref{prop:bip-hamiltonian} that $G_p$ is
Hamiltonian.
\end{proof}

\section{Hamiltonicity game on Dirac graphs}

In this section we prove Theorem \ref{thm_maingame}.
We begin by presenting some standard results and techniques in
positional game theory which we will need later.
In 1973, Erd\H{o}s and Selfridge \cite{ErSe} gave a
sufficient condition for Breaker's win in a $(1:1)$ Maker-Breaker game.
Later, Beck \cite{Beck} generalized this result and
proved the following theorem.
\begin{thm} \label{thm:beck}
Consider a $(p:q)$ Maker-Breaker game played over some board
where $\mathcal{F}$ is the family of winning sets. If
\[ \sum_{B \in \mathcal{F}} (1 + q)^{-|B|/p} < \frac{1}{1+q}, \]
then Breaker has a winning strategy.
\end{thm}

Suppose that we are playing a $(1:b)$ Maker-Breaker game over a
board $V$ with the family $\mathcal{F}$ of winning sets, and let $a$ be some fixed
integer. It is sometimes convenient to partition the board into $a$
boards $V_1 \cup \cdots \cup V_a$ and to play a $(1:ab)$ game on
each board $V_i$ separately. That is, Maker will start by playing
the board $V_1$, and after playing board $V_i$, in the next round
will play board $V_{i+1}$ (index addition is modulo $a$). Note that
after Maker plays the board $V_1$ for the first time, and until
playing it for the second time, Breaker can claim at most $ab$
elements of $V_1$. Therefore Maker may assume to his/her
disadvantage that he/she is playing a $(1:ab)$ game on each board as
the second player. If one shows that Maker can claim elements so
that certain properties are satisfied for each board, then by
combining these properties, we may show in the end that the Maker's
elements altogether contain some winning set. When we say that we
{\em split the board}, we suppose that we partitioned the board into
some fixed number of boards as above.

We will use later the following concentration result (see, e.g.,
\cite[Theorem 2.10]{JaLuRu}). Let $A$ and $A'$ be sets such that $A'
\subset A$, and $|A| = N$, $|A'| = m$. Let $B$ be a subset of $A$ of
size $n$ chosen uniformly at random. Then the distribution of the
random variable $|B \cap A'|$ is called the \textit{hypergeometric
distribution with parameters $N,n$, and $m$}.

\begin{thm} \label{lemma_hypergeometric}
Let $\varepsilon$ be a fixed positive constant and let $X$ be a
random variable with hypergeometric distribution with parameters
$N,n$, and $m$. Then for all $t \ge 0$,
\[ P\big(|X - \mathbb{E}[X]| \geq t\big) \leq 2e^{-2t^2/n}. \]
\end{thm}

Let $G$ be a Dirac graph, and as in the previous section, we begin
by applying Lemma \ref{lem:dirac} with $\alpha=\frac{1}{2^{40}}$ and
$\gamma=\frac{1}{96}$ to classify Dirac graphs into three types. We
will show case by case that Maker can win a $(1:b)$ Hamiltonicity
Maker-Breaker game played on $G$ if $b \le \frac{cn}{\log n}$ for
some small positive constant $c$.

\subsection{First case}

We first assume that $e(A,B)\ge \alpha n^{2}$ for all
half-sets $A$ and $B$ as in (i) of Lemma \ref{lem:dirac}.

\begin{lem} \label{lem_game_case1_lem1}
There exist positive reals $\varepsilon, c$, and $\beta$
the following holds for every $b \le \frac{cn}{\log n}$.
In a $(1:b)$ Maker-Breaker game played on $G$,
Maker can construct a graph $M_1$ with
property $\mathcal{RE}(\frac{1}{2}+\varepsilon)$ in the first $\beta n \log n$ rounds.
\end{lem}
\begin{proof}
Let $\varepsilon = \min\{\varepsilon_{\ref{lem_halfexpander}},\left(\frac{\alpha}{8}\right)^5\}$
and $C = 16\varepsilon^{-3}$. Let $c = \min\{\frac{1}{42C},
\frac{\alpha}{30}\}$, $\beta = \frac{1}{2c}$, and $b_0 =
\frac{cn}{\log n}$. We will show that in a $(1:b_0)$ Maker-Breaker
game played on $G$, Maker can construct a graph with property
$\mathcal{RE}(\frac{1}{2}+\varepsilon)$. If this indeed is true,
then since at most $\beta n \log n$ rounds can be played, we can see
that for all $b \le b_0$, in a $(1:b)$ Maker-Breaker game played on
$G$, Maker can construct a graph $M_1$ with property
$\mathcal{RE}(\frac{1}{2}+\varepsilon)$ in the first $\beta n \log
n$ rounds. This is because Maker can always trick himself/herself
that he/she is playing a $(1:b_0)$ game, by arbitrarily adding
$b_0-b$ `fake' Breaker edges. If Breaker later happens to claim some
fake Breaker edge, then Maker will replace that fake edge with
another fake edge which has not yet been claimed. In this way, Maker
is essentially playing a $(1:b_0)$ game, and thus can construct a
graph $M_1$ with property $\mathcal{RE}(\frac{1}{2}+\varepsilon)$ in
the first ${n \choose 2}/b_0 \leq \beta n \log n$ rounds.

Let $r = C \log n$ and note that by Lemma \ref{lem_halfexpander},
it suffices to show that Maker can construct a half-expander with
parameters $\varepsilon$ and $r$. A naive approach directly using
Beck's criteria (Theorem \ref{thm:beck}) fails for our range of
bias, and thus we use the strategy of Krivelevich and Szab\'o
\cite{KrSz}. In this strategy, we construct an auxiliary hypergraph
whose vertex set is the edge set of $G$, and play a Maker-Breaker
game on this hypergraph. The board is the vertex set of the
hypergraph, and the winning sets are the edges of the hypergraph (to
avoid confusion, we name the players of this game as NewMaker and
NewBreaker). Maker (resp. Breaker) of the original game will play
NewBreaker (resp. NewMaker) in the auxiliary game. By doing so, we
wish to establish the fact that Maker can strategically claim edges
so that Maker's graph satisfies the conditions of a half-expander.

For each $\ell=1, \cdots, n$, let $V_{\ell, 1}, \cdots, V_{\ell,
\ell}$ be fixed disjoint vertex subsets of $V(G)$ of size $\lfloor
\frac{n}{\ell}\rfloor$, and for each index subset $J \subset
[\ell]$, let $V_{\ell, J} = \cup_{j \in J} V_{\ell,j}$. Let $H_1,
H_2$ and $H_3$ be hypergraphs that have the edge set of $G$ as their
vertex set. The edge set of $H_1$ is constructed as follows: for
each set of vertices $X \subset V(G)$ of size $|X| = i \le
\frac{\varepsilon n}{r}$ and each index set $J \subset [3ri]$ of size $|J|
= 2ri$, place a hyperedge consisting of the edges of $G$ that have
one endpoint in $X$, and the other endpoint in $V_{3ri,J} \setminus
X$. Since $G$ has minimum degree at least $\frac{n}{2}$, the size of
a hyperedge constructed this way is at least
\[|X| \cdot \left(|J|\cdot |V_{3ri,1}| - \frac{n}{2} - |X| \right)
\ge i \cdot \left(2ri \Big(\frac{n}{3ri} - 1\Big) - \frac{n}{2} - i \right)
\ge \left(\frac{1}{6} - o(1)\right) ni \ge \frac{ni}{7},\]
and the number of hyperedges of $H_1$ constructed from
subsets of vertices of size $i$ is ${n \choose i}{3ri \choose 2ri}$.
Assume that Maker claims at least one edge (vertex of the hypergraph)
in each of the hyperedges as above. Then it follows that for each set
of vertices $X$ of size at most $\frac{\varepsilon n}{r}$,
Maker's graph has $|N(X)| > ri = r|X|$, as otherwise, we can find
an index set $J \subset [3ri]$ of size $|J| = 2ri$ such that there are
no edges which have one endpoint in $X$ and the other endpoint in $V_{3ri, J}$.
Thus in such a situation, Maker's graph will satisfy Condition (i) of
Definition \ref{def:expander}.

The edge set of $H_2$ is constructed as follows: for each pair
of sets of vertices $X,Y \subset V(G)$
of sizes $|X| = \frac{n}{\varepsilon r}$ and
$|Y| = (\frac{1}{2} + \varepsilon)n$, place a hyperedge
consisting of the edges of $G$
that have one endpoint in $X$, and the other endpoint in
$Y \setminus X$. Since $G$ has minimum degree at least $\frac{n}{2}$,
the size of a hyperedge constructed this way is at least
\[|X| \cdot \left(|Y| - \frac{n}{2} - |X| \right)
\ge \frac{n}{\varepsilon r} \cdot \left(\varepsilon n - \frac{n}{\varepsilon r} \right)
\ge \frac{n^2}{2r}, \]
and the number of such hyperedges is at most $2^{2n}$.
Moreover, if Maker can claim at least one
edge in each of the hyperedges of $H_2$, then Maker's graph will satisfy
Condition (ii) of Definition \ref{def:expander}.

For every pair of disjoint sets
of vertices $X$ and $Y$ such that $|X|,|Y| \ge (\frac{1}{2} - \varepsilon^{1/5})n$,
let $E_{X,Y}$ be the set of edges that have one endpoint in $X$ and the other
endpoint in $Y$ (edges within $X\cap Y$ are counted once). 
The edge set of $H_3$ is as follows: for every $X$ and $Y$ as
above, place a hyperedge over every subset of $E_{X,Y}$ of size exactly
$|E_{X,Y}| - 2n$. Since 
$\frac{3}{4}\alpha n^2\le \alpha n^2-2\varepsilon^{1/5}n^2 \le e_{G}(X,Y) \le n^2$, 
each such hyperedge has size at least 
$|E_{X,Y}| \ge \frac{1}{2}e_G(X,Y) - 2n \ge \frac{\alpha n^2}{3}$, and the total
number of hyperedges of $H_3$ is at most
\[ 2^{2n} \cdot {n^2 \choose 2n} \le 2^{2n} \cdot n^{4n} \le e^{5n\log n}. \]
Moreover, if Maker can claim at least one
edge in each of the hyperedges of $H_3$, then Maker's graph will satisfy
Condition (iii) of Definition \ref{def:expander}.

Let $H = H_1 \cup H_2 \cup H_3$ and consider a $(b_0:1)$
Maker-Breaker game played on the hypergraph $H$ (where we name the
players as NewMaker and NewBreaker in order to distinguish the
players from our game). By the arguments above, it suffices to
show that the auxiliary game is NewBreaker's win in order to
establish our lemma. This will be done by using Theorem
\ref{thm:beck} with $p=b_0$ and $q=1$. We thus would like to show
that
\begin{eqnarray}
\label{eq:keybeck}
\sum_{i=1}^{3} \sum_{e \in E(H_i)} 2^{-|e|/b_0} = \sum_{i=1}^{3} \sum_{e \in E(H_i)} 2^{-|e|\log n/(cn)}
\end{eqnarray}
is at most $\frac{1}{2}$. For the hypergraph $H_1$, using the fact
$c \le \frac{1}{42C}$, we have
\begin{align*}
\sum_{e \in E(H_1)} 2^{-|e|/b_0} &\le \sum_{i=1}^{\varepsilon n/r} {n \choose i}{3ri \choose 2ri} 2^{-ni\log n/(7cn)}
\le \sum_{i=1}^{\varepsilon n/r} n^{i} \cdot 2^{3ri} \cdot  2^{-i\log n/(7c)} \\
&\le \sum_{i=1}^{\varepsilon n/r} e^{(3C+1 - \frac{1}{7c}) i \log n } \le \frac{1}{4}.
\end{align*}
For hypergraphs $H_2$ and $H_3$, using the fact $c \le
\min\{\frac{1}{42C},\frac{\alpha}{30}\}$, we have
\begin{align*}
\sum_{e \in E(H_2)} 2^{-|e|/b_0} + \sum_{e \in E(H_3)} 2^{-|e|/b_0}
  &\le 2^{2n} \cdot 2^{-n^2 \log n/(2rcn)} + e^{5n\log n} \cdot 2^{-\alpha n^2 \log n/(3cn)} \\
  &\le 2^{2n} \cdot 2^{-n/(2cC)} + e^{5n\log n} \cdot 2^{-\alpha n\log n/(3c)} \\
  &\le \frac{1}{4}.
\end{align*}
Therefore, by combining the two inequalities, we get $\eqref{eq:keybeck} \le \frac{1}{2}$.
\end{proof}

\begin{lem} \label{lem_game_case1_lem2}
For every fixed positive reals $\varepsilon$ and $\beta$, there
exists a positive constant $c = c(\varepsilon, \beta)$ satisfying
the following for every $b \le \frac{cn}{\log n}$. In a $(1:b)$
Maker-Breaker game played on $G$, if Maker constructs a graph $M_1$
with property $\mathcal{RE}(\frac{1}{2} + \varepsilon)$ in the first
$\beta n \log n$ rounds, then he in the remaining rounds can
construct a graph $M_2$ which complements $M_1$.
\end{lem}
\begin{proof}
Let $c \le \min\{c_{\ref{lem_game_case1_lem1}}(\varepsilon),
\frac{\varepsilon^2}{2\beta},\frac{\varepsilon}{8}\}$. Note that since we played $\beta
n\log n$ rounds to construct $M_1$, the number of edges claimed 
so far is at most $c \beta n^2 + \beta n\log n\le 2c\beta n^2 \le \varepsilon^2n^2$. Let $P$ be a path with a
fixed edge $e$ such that there is no path longer than $P$ containing
$e$ in the graph $P \cup M_1$. Then there exists a set $S_P \subset
V(P)$ of size $|S_P| \ge (\frac{1}{2} + \varepsilon)n$ such that for
every $v \in S_P$ there exists a set $T_v$ of size $|T_v| \ge
(\frac{1}{2} + \varepsilon)n$ such that for all $w \in T_v$, there
exists a path of the same length as $P$ containing $e$, starting at
$v$ and ending at $w$.

Since $G$ has minimum degree at least $\frac{n}{2}$, we know that
for each vertex $v \in S_P$, at least $\varepsilon n$ vertices in
$T_v$ form an edge with $v$ in the graph $G$. Since at most 
$\varepsilon^2 n^2$ edges have been claimed so
far, in total we have at least
\[ \frac{1}{2} \cdot \left(\left(\frac{1}{2} + \varepsilon \right)n \cdot \varepsilon n - \varepsilon^2 n ^2\right) \ge \frac{\varepsilon}{4}n^2 \]
edges such that if Maker can claim at least one of these edges, then
he can extend $P$. Consequently, if Maker can do this for all
paths $P$, then we prove our lemma (the factor $\frac{1}{2}$ comes
from the fact that the same pair $(v,w)$ can be counted at most twice,
once as $v \in S_P, w \in T_w$ and once as $w \in S_P, v \in T_w$).

There are at most $n^2 \cdot n!$ paths that we need to consider, and
for each path we have $\frac{\varepsilon}{4}n^2$ edges where Maker
has to claim at least one of these edges. Consider the
following Maker-Breaker game (where we name the players as NewMaker
and NewBreaker in order to distinguish the players from our game).
The board is defined as the edges which have not been claimed in the
first $\beta n\log n$ rounds. The winning sets are defined as sets
of at least $\frac{\varepsilon}{4}n^2$ edges for each non-extendable
path $P$ with a fixed edge which we described above. Note that there
are at most $n^2 \cdot n!$ winning sets. It suffices to show that
NewBreaker wins this new game, since our Maker will play
NewBreaker's role here (thus he/she wants to claim at least one edge
from each of the winning sets). We can use Beck's criterion, Theorem
\ref{thm:beck}, with $p = \frac{cn}{\log n}$ and $q=1$ to see that
the newly defined game is indeed NewBreaker's win since $c\le \frac{\varepsilon}{8}$:
\[
\sum_{B \in \mathcal{F}} 2^{-|B|/p} \le  n^2 \cdot n! \cdot 2^{-(\varepsilon/4)n^2/(cn/\log n)} < e^{n\log n}
2^{-(\varepsilon/(4c))n \log n} < \frac{1}{2}. \]
\end{proof}

\begin{thm}
There exists a constant $c$ such that the following holds for every
$b \le \frac{cn}{\log n}$. If $G$ is a Dirac graph satisfying $(i)$
of Lemma \ref{lem:dirac}, then Maker has a winning strategy for the
$(1:b)$-Maker-Breaker Hamiltonicity game.
\end{thm}
\begin{proof}
Let $\varepsilon = \varepsilon_{\ref{lem_game_case1_lem1}}$,
$\beta = \beta_{\ref{lem_game_case1_lem1}}$
and let $c = \min\{c_{\ref{lem_game_case1_lem1}},  c_{\ref{lem_game_case1_lem2}}(\varepsilon, \beta) \}$.
By Lemma \ref{lem_game_case1_lem1},
Maker can construct a graph $M_1$ with property $\mathcal{RE}(\frac{1}{2} +
\varepsilon)$ in the first $\beta n \log n$ rounds. Then by Lemma
\ref{lem_game_case1_lem2}, Maker can construct a graph $M_2$ which
complements $G_1$ in the remaining rounds. Therefore by Proposition
\ref{prop:rotationextension}, Maker can construct a Hamilton cycle
and win the game.
\end{proof}

\subsection{Second case}

We assume that there exists a set $A$ of size $\frac{n}{2} \le|A|\le
(\frac{1}{2}+16\alpha)n$ such that $e(A,\overline{A})\le 6\alpha
n^{2}$, and the induced subgraphs on both $A$ and $\overline{A}$
have minimum degree at least $\frac{n}{5}$, as in (ii) of Lemma
\ref{lem:dirac}. The same computation as in \eqref{eq:lem_case2}
shows that
\begin{align} \label{eq:lem_game_case2}
e(A) \ge {|A| \choose 2} - 11\alpha n^2
\quad \textrm{and} \quad e(\overline{A}) \ge {|\overline{A}| \choose
2} - 6\alpha n^2.
\end{align}

\begin{lem} \label{lem_game_case2_lem1}
There exist positive reals $c$ and $\beta$ satisfying the following
for every positive $b \le \frac{cn}{\log n}$. In a $(1:b)$
Maker-Breaker game played on the board $G[A]$, Maker can construct a
graph with property $\mathcal{RE}(\frac{1}{6})$ in the first $\beta
n \log n$ rounds (similar for $G[\overline{A}]$).
\end{lem}
\begin{proof}
Let $c$ be a small enough constant depending on $\alpha$. Let $\beta
= \frac{1}{2c}$ and $b_0 = \frac{cn}{\log n}$. We will show that in
a $(1:b_0)$ Maker-Breaker game played on $G[A]$, Maker can construct
a graph with property $\mathcal{RE}(\frac{1}{2}+\varepsilon)$
(similar proof can be used to establish the statement for
$G[\overline{A}]$). As in Lemma \ref{lem_game_case1_lem1}, this will
imply that for all $b \le b_0$, in a $(1:b)$ Maker-Breaker game
played on $G$, Maker can construct a graph $M_1$ with property
$\mathcal{RE}(\frac{1}{6})$ in the first $\beta n \log n$ rounds.
Note that by Lemma \ref{lem:expander_rotationextension}, it
suffices to show that Maker can construct a half-expander with
parameters $\frac{1}{4}$ and $r = 2^{20}$. We will construct an
auxiliary hypergraph whose vertex set is the edge set of $G$, and
play a Maker-Breaker game on this hypergraph (we name the players of
this game as NewMaker and NewBreaker). Maker (resp. Breaker) of the
original game will play NewBreaker (resp. NewMaker) in the auxiliary
game.

For each $\ell=1, \cdots, n$, let $V_{\ell, 1}, \cdots, V_{\ell,
\ell}$ be fixed disjoint vertex subsets of $V(G)$ of size $\lfloor
\frac{|A|}{\ell}\rfloor$, and for each index subset $J \subset
[\ell]$, let $V_{\ell, J} = \cup_{j \in J} V_{\ell,j}$. Let $H_1$
and $H_2$ be hypergraphs that have the edge set of $G$ as their
vertex set. The edge set of $H_1$ is constructed as follows: for
each set of vertices $X \subset V(G)$ of size $|X| = i \le
\frac{n}{r^{3/2}}$ and each index set $J \subset [10ri]$ of size
$|J| = 9ri$, place a hyperedge consisting of the edges of $G$ that
have one endpoint in $X$, and the other endpoint in $V_{10ri,J}
\setminus X$. Since $G$ has minimum degree at least $\frac{n}{5}$,
the size of a hyperedge constructed this way is at least
\[|X| \cdot \left(|J|\cdot |V_{10ri,1}| - (|A| - \frac{n}{5}) - |X| \right)
\ge i \cdot \left(9ri \Big(\frac{|A|}{10ri} - 1\Big) -
\frac{4|A|}{5} - i \right) \ge \frac{|A|i}{11},\] and the number of
hyperedges of $H_1$ constructed from subsets of vertices of size
$i$ is ${|A| \choose i}{10ri \choose 9ri}$. Assume that Maker claims
at least one edge (vertex of the hypergraph) in each of the
hyperedges as above. Then it follows that for each set of vertices
$X$ of size at most $\frac{|A|}{r^{3/2}}$, Maker's graph has $|N(X)|
> ri = r|X|$, as otherwise, we can find an index set $J \subset
[10ri]$ of size $|J| = 9ri$ such that there are no edges which have
one endpoint in $X$ and the other endpoint in $V_{10ri, J}$. Thus in
such a situation, Maker's graph will satisfy Condition (i) of
Definition \ref{def:realexpander}.

The edge set of $H_2$ is constructed as follows: for each pair of
sets of vertices $X,Y \subset V(G)$ of sizes $|X| =
\frac{|A|}{r^{3/4}}$ and $|Y| = \frac{|A|}{4}$, place a hyperedge
consisting of the edges of $G$ that have one endpoint in $X$, and
the other endpoint in $Y \setminus X$. By \eqref{eq:lem_game_case2},
the size of a hyperedge constructed this way is at least
\[ |X|\cdot (|Y|-|X|) - 11\alpha n^2
\ge \frac{|A|}{2^{15}} \cdot \left(\frac{|A|}{4} -
\frac{|A|}{2^{15}} \right) - 11\alpha n^2 \ge \frac{|A|^2}{2^{20}},
\] and the number of such hyperedges is at most $2^{2n}$. Moreover,
if Maker can claim at least one edge in each of the hyperedges of
$H_2$, then Maker's graph will satisfy Condition (ii) of Definition
\ref{def:realexpander}.

Let $H = H_1 \cup H_2$ and consider a $(b_0:1)$ Maker-Breaker game
played on the hypergraph $H$ (where we name the players as NewMaker
and NewBreaker in order to distinguish the players from our game).
By the observations above, it suffices to show that the auxiliary
game is NewBreaker's win in order to establish our lemma. This can
be done by using Theorem \ref{thm:beck} with $p=b_0$ and $q=1$,
given that $c$ is small enough. We omit the detailed computation.
\end{proof}

\begin{lem}  \label{lem_game_case2_lem2}
For every fixed positive real $\beta$, there exists a constant $c =
c(\beta)$ such that the following holds for every positive $b \le
\frac{cn}{\log n}$. In a $(1:b)$ Maker-Breaker game played on the
board $G[A]$, suppose that Maker constructed a graph with property
$\mathcal{RE}(\frac{1}{6})$ in the first $\beta n \log n$ rounds.
Then he can construct a graph that complements it in the
remaining rounds (similar for $G[\overline{A}]$).
\end{lem}
\begin{proof}
We will only prove the statement for $G[A]$, since the statement for
$G[\overline{A}]$ can be proved similarly. Let $c \le \alpha
\beta^{-1}$ be a small enough constant. Let $M_1$ be the Maker's
graph constructed in the first $\beta n \log n$ rounds. Note that
the number of edges that have been claimed in the first $\beta n
\log n$ rounds is at most $c \beta n^2 + \beta n \log n \le 2c\beta
n^2$. Let $G'$ be the graph of the edges that have not been
claimed by Breaker so far. Let $P$ be a path over a subset of
vertices of $A$, with a fixed edge $e$ such that there is no path
longer than $P$ containing $e$ in the graph $P \cup M_1$. Then there
exists a set $S_P \subset P$ of size $|S_P| \ge \frac{|A|}{6}$ such
that for every $v \in S_P$ there exists a set $T_v$ of size $|T_v|
\ge \frac{|A|}{6}$ such that for all $w \in T_v$, there exists a
path of the same length as $P$ containing $e$, starting at $v$ and
ending at $w$. By \eqref{eq:lem_game_case2}, we have \[ e_{G'}(S_P,
T_v) \ge |S_P|\cdot |T_v| - 11\alpha n^2 - 2c\beta n^2 \ge
\frac{|A|^2}{36} - 13 \alpha n^2 \ge \frac{n^2}{200}.
\] Using this estimate, we can proceed as in Lemma
\ref{lem_game_case1_lem2} to finish the proof. We omit the details.
\end{proof}

By using the two lemmas above, we can show that Maker can win the Hamiltonicity
game in this case as well.

\begin{thm}
There exists a constant $c$ such that the following holds for every
positive $b \le \frac{cn}{\log n}$. If $G$ is a Dirac graph
satisfying $(ii)$ of Lemma \ref{lem:dirac}, then Maker has a winning
strategy for the $(1:b)$-Maker-Breaker Hamiltonicity game.
\end{thm}
\begin{proof}
Let $\beta = \beta_{\ref{lem_game_case2_lem1}}$, and $c =
\frac{1}{3}\min\{ c_{\ref{lem_game_case2_lem1}},
c_{\ref{lem_game_case2_lem2}}(\beta)\}$. We split the board into
three boards, $G[A], G[\overline{A}]$, and the subgraph induced by
the edges between $A$ and $\overline{A}$ (call the last one the
bipartite board $\mathcal{B}$). To Maker's disadvantage, we will
separately play a $(1:3b)$ game on each board, and show that Maker
can play the three games so that in the end, the union of Maker's
graphs in the three boards is Hamiltonian. By Lemma
\ref{lem_game_case2_lem1}, we know that Maker can construct
subgraphs of $G[A]$ and $G[\overline{A}]$ which have property
$\mathcal{RE}(\frac{1}{6})$ in the first $\beta n\log n$ rounds of
each board. Then by Lemma \ref{lem_game_case2_lem2}, Maker can
construct graphs that complement these subgraphs in the remaining
rounds of each board. Therefore, by Proposition
\ref{prop:rotationextension}, Maker can construct subgraphs of
$G[A]$ and $G[\overline{A}]$ which are Hamiltonian connected. Thus
it suffices to show that Maker can claim two vertex disjoint edges
in the bipartite board, since together with the Hamilton
connectivity of Maker's graph in $G[A]$ and $G[\overline{A}]$, this
will imply Hamiltonicity of Maker's graph.

Consider the bipartite board $\mathcal{B}$. We will consider two
cases depending on the sizes of $A$ and $\overline{A}$. First assume
that $|A| = |\overline{A}| = \frac{n}{2}$. Then since the minimum
degree of $G$ is at least $\frac{n}{2}$, the graph $\mathcal{B}$ has
minimum degree at least 1. If there is no vertex in $A$ of degree at
least $\frac{n}{3}$ in $\mathcal{B}$, then Maker starts by claiming
an arbitrary edge $\{v,w\}$ of $\mathcal{B}$ such that $v \in A, w
\in \overline{A}$. Breaker can then claim at most $\frac{2cn}{\log
n}$ other edges before Maker's next move since Breaker might have
been the first player. Afterwards, since $\mathcal{B}$ has minimum
degree at least 1, we can see that there exists at least
$\frac{n}{2} - 1 - \frac{2cn}{\log n}$ vertices other than $w$ in
$\overline{A}$ which have at least 1 non-claimed edge incident to
it. Among them, at most $\frac{n}{3}$ can be incident to $v$.
Therefore Maker can claim an edge $\{v',w'\}$ such that $v \neq v'$
and $w \neq w'$. Similarly, we can take care of the case when
$\overline{A}$ has no vertex of degree at least $\frac{n}{3}$. Thus
we may assume that there exist vertices $v_0 \in A$  and $w_0 \in
\overline{A}$ such that $v_0$ and $w_0$ have degree at least
$\frac{n}{3}$. In this case, Maker in the first round claims an edge
incident to $v_0$ which is not $\{v_0, w_0\}$, and in the second
round claims an edge incident to $w_0$ which is not $\{v_0, w_0\}$
(this can be done since Breaker cannot claim all the edges incident to
$w_0$ in two rounds). Thus Maker can claim two vertex-disjoint edges
in this case.

Second, assume that $|A| = \frac{n}{2} + t$ for $t>0$. Then all the
vertices of $\overline{A}$ have degree at least $\lceil t  + 1
\rceil \ge 2$ in $\mathcal{B}$. Maker starts by claiming an
arbitrary edge $\{v,w\}$. Note that since all the vertices in
$\overline{A}$ have degree at least 2, there are at least  \[
|\overline{A}|-1 \ge \Big(\frac{1}{2} - 16\alpha \Big)n - 1 \ge
\frac{n}{3} \] edges remaining which are not incident to $\{v,w\}$.
Breaker cannot claim all of these edges in two rounds, and thus
Maker can claim one such edge in the next round to achieve his/her
goal. This concludes the proof.
\end{proof}

\subsection{Third case}

We assume that there exists a set $A$ of size $\frac{n}{2}
\le|A|\le(\frac{1}{2} +16\alpha)n$ such that the bipartite graph
induced by the edges between $A$ and $\overline{A}$ has at least
$(\frac{1}{4}-14\alpha)n^{2}$ edges and minimum degree at least
$\frac{n}{192}$. Moreover, either $|A|=\lceil \frac{n}{2}\rceil$, or
the induced subgraph $G[A]$ has maximum degree at most
$\frac{n}{96}$. Let $V_1 = A$ and $V_2 = \overline{A}$. Let
$k=|V_1|-|V_2|$ so that $|V_1| = \frac{n+k}{2}$, $|V_2| =
\frac{n-k}{2}$, and $k \le 32\alpha n$.

\begin{lem} \label{lem_game_case3_lem1}
There exist positive reals $c$ and $\beta$ such that the following
holds for every $b \le \frac{cn}{\log n}$. In a $(1:b)$
Maker-Breaker game played on $G$, Maker can construct a graph with
property $\BIPRE(\frac{1}{6})$ in the first $\beta n \log n$ rounds.
\end{lem}
\begin{proof}
Let $\beta$ be a large enough positive constant. Let $\varepsilon =
\min\{\frac{1}{150 \beta}, \frac{1}{240}\}$ and $c \le
\frac{\varepsilon}{2\beta}$ be small enough. We start by splitting
the board into two boards $G[V_1]$, and the bipartite graph induced
by the edges between $V_1$ and $V_2$ (which we refer to as the {\em
bipartite board} $\mathcal{B}$). We first show that on the board
$G[V_1]$, Maker can claim at least $2k$ vertex disjoint edges. We
only need to consider the cases when $k \ge 1$ since otherwise the
claim is trivial. Suppose that Maker has claimed $t$ vertex disjoint
edges after $t$ rounds for some $t < 2k$. Note that since $|V_2| =
\frac{n-k}{2}$, the induced subgraph $G[V_1]$ has minimum degree at
least $\frac{k}{2}$. Since the graph $G[V_1]$ has maximum degree at
most $\frac{n}{96}$, the number of non-claimed edges not incident to
any of the $t$ edges that Maker has claimed so far is at least
\begin{align} \label{eq:freespecialedge}
\frac{1}{2} \cdot \frac{k}{2} \cdot \frac{n}{2} - 2t \cdot \frac{n}{96} - (t+1) \cdot \frac{cn}{\log n} \ge \frac{nk}{24}.
\end{align}
Maker chooses an arbitrary edge out of these edges. In the end,
Maker can claim at least $2k$ vertex-disjoint edges in the board $G[V_1]$.

On the board $\mathcal{B}$, one can show that there exists a
constant $C$ which goes to infinity as $c$ goes to zero, such that
for $s= C\log n$, Maker can construct a graph satisfying the
following:
\begin{enumerate}
  \setlength{\itemsep}{1pt} \setlength{\parskip}{0pt}
  \setlength{\parsep}{0pt}
\item For all $X\subset V_{1}$,
\begin{enumerate}
\item if $|X|\le \frac{n}{s^{3/2}}$, then $|N(X)\cap V_{2}|\ge s|X|$,
\item if $\frac{n}{s^{3/4}} \le |X| \le \frac{n}{2^{30}}$, then $|N(X)\cap V_{2}|\ge
\frac{n}{200}$,
\item if $\frac{n}{2^{30}} \le |X|$, then $|N(X)\cap V_{2}|\ge \frac{7}{8}|V_2|$.
\end{enumerate}
\item For all $Y\subset V_{2}$,
\begin{enumerate}
\item if $|Y|\le \frac{n}{s^{3/2}}$, then $|N(Y)\cap V_{1}|\ge s|Y|$,
\item if $\frac{n}{s^{3/4}} \le |Y| \le \frac{n}{2^{30}}$, then $|N(Y)\cap V_{1}|\ge
\frac{n}{200}$,
\item if $\frac{n}{2^{30}} \le |Y|$, then $|N(Y)\cap V_{1}|\ge \frac{7}{8}|V_1|$.
\end{enumerate}
\end{enumerate}
To prove $1(a),1(b),2(a),2(b)$, we can use the fact that the
bipartite board has minimum degree at least $\frac{n}{192}$, and to
prove $1(c),2(c)$, we can use the fact that $e(A, \overline{A}) \ge
(\frac{1}{4}-14\alpha)n^2$. We omit the details which are similar to
those of Lemmas \ref{lem_game_case1_lem1} and
\ref{lem_game_case2_lem1}.

It then suffices to show that one can carefully choose $k$ of the
edges within $G[V_1]$ as the special edges so that we can find a
special frame, with respect to which Maker's graph is a
bipartite-expander with certain parameters. Recall that Maker
claimed at least $2k$ vertex-disjoint edges on the board $G[V_1]$.
Call these {\em candidate edges}. Uniformly at random choose $k$
edges among the candidate edges, and declare them as our special
edges $S_E$. For each such edge, declare one of its vertices as a
special vertex (let $S_V$ be the set of special vertices). This
forms a special frame $(V_1, V_2, S_V, S_E)$ of Maker's graph. We
now verify that Maker's graph is a bipartite-expander with
parameters $\epsilon=\frac{1}{4}$ and $r=2^{20}$ with respect to this frame.

Condition (i) of Definition~\ref{def:bip-expander} easily follows
from 1(a), 1(b), and 1(c) given above. Let $V_1''$ be as in
Definition \ref{def:almostbipartite}. Fix a set $Y \subset V_2$ of
size at most $\frac{n}{s^{3/2}}$, and note that $|N(Y) \cap V_1| \ge
s|Y|$. We can find a subset $N_Y$ of $N(Y) \cap V_1$ of size at
least $\frac{s}{2}|Y|$ such that each candidate edge intersects
$N_Y$ in at most one vertex. Note that if $|N_Y \cap V(S_E)| < \frac{s}{3}|Y|$,
then $|N(Y)\cap V_1''|\ge |N_Y \setminus V(S_E)| \ge \frac{s}{6}|Y|$.
Since we randomly chose the special
edges and the probability that each edge is chosen is at most
$\frac{1}{2}$, the probability that $|N_Y \cap V(S_E)| \ge
\frac{s}{3}|Y|$ is, by the concentration of hypergeometric
distribution, at most $e^{-\Omega(s|Y|)}$. By taking the union bound
over all possible choices of $Y$, we see that the probability that
there exists a set $Y$ for which $|N_Y \cap V(S_E)|\ge \frac{s}{3}|Y|$ 
is (for small enough $c$, and thus large enough $C$) at most
\begin{align*}
\sum_{i=1}^{n/s^{3/2}} {n \choose i} e^{-\Omega(is)}
  \le \sum_{i=1}^{n/s^{3/2}} n^{i} n^{-\Omega(Ci)}
  \le \sum_{i=1}^{n/s^{3/2}} n^{-i} = o(1).
\end{align*}
Consequently, we can choose $S_V$ and $S_E$ so that $|N(Y) \cap V_1''| \ge \frac{s}{6}|Y|$
for all $Y$.
Consider a set $Y \subset V_2$  of size between $\frac{n}{s^{3/2}}$
and $\frac{n}{s^{3/4}}$. For an arbitrary subset $Y'\subset Y$ of size $\frac{n}{s^{3/2}}$,
as we explained above,
$$|N(Y)\cap V_1''| \geq |N(Y')\cap V_1''| \geq  \frac{s}{6}|Y'|=\frac{n}{6s^{1/2}} \geq
\frac{s^{1/4}}{6}|Y|>2^{20}|Y|.$$
In the case $Y \subset V_2$ is of size between $\frac{n}{s^{3/4}}$
and $\frac{n}{2^{30}}$, we use the fact $|N(Y) \cap V_1| \ge
\frac{n}{200}$. In this case, using that $\alpha=2^{-40}$, we have that
$$|N(Y) \cap V''_1|\geq |N(Y) \cap V_1|-2k\ge \frac{n}{200}-64\alpha n \geq  \frac{n}{300} \geq
2^{20}|Y|.$$
Therefore the first part of Condition (ii) of
Definition~\ref{def:bip-expander} holds. To establish the second
part, let $Y \subset V_2$ be a set of size at least
$\frac{n}{2^{15}}$. For an arbitrary choice of $S_E$, we have
\[ |N(Y) \cap V_1''| \ge |N(Y) \cap V_1| - 2k \ge \frac{7}{8}|V_1| - 64\alpha n \ge \frac{3}{4}|V_1|,  \]
and this gives the second part of Condition (ii) of
Definition~\ref{def:bip-expander} as well. Thus Maker's graph is a
bipartite-expander with parameters $\frac{1}{4}$ and $2^{20}$, and
by Lemma \ref{lem:bipexpander_rotationextension}, it follows that
Maker's graph has property $\BIPRE(\frac{1}{6})$.
\end{proof}

\begin{lem}  \label{lem_game_case3_lem2}
For every positive real $\beta$, there exists a positive constant $c
= c(\beta)$ such that the following holds for every $b \le
\frac{cn}{\log n}$. In a $(1:b)$ Maker-Breaker game played on $G$,
if Maker constructs a graph with property $\BIPRE(\frac{1}{6})$ in
the first $\beta n \log n$ rounds, then he can construct in the
remaining rounds a graph that complements it.
\end{lem}
\begin{proof}
Let $c \le \alpha \beta^{-1}$ be a small enough constant. Let $M_1$
be the Maker's graph constructed in the first $\beta n \log n$
rounds. Note that the number of edges that have been claimed in the
first $\beta n \log n$ rounds is at most $c \beta n^2 + \beta n \log
n \le 2c\beta n^2$. Let $G'$ be the graph induced by the edges that
have not been claimed by Breaker so far. Let $P$ be a proper path
such that there is no path longer than $P$ in the graph $P \cup
M_1$. Then there exists a set $S_P \subset V_2$ of size $|S_P| \ge
\frac{n}{6}$ such that for every $v \in S_P$ there exists a set
$T_v \subset V_1$ of size $|T_v| \ge \frac{n}{6}$ such that for all $w \in
T_v$, there exists a path of the same length as $P$ starting at $v$
and ending at $w$. By the fact $e(A, \overline{A}) \ge (\frac{1}{4}
- 14\alpha)n^2 \ge |A|\cdot|\overline{A}| - 14\alpha n^2$, we have
\[ e_{G'}(S_P, T_v) \ge |S_P|\cdot |T_v| - 14\alpha n^2 - 2c\beta
n^2 \ge \frac{n^2}{36} - 16 \alpha n^2 \ge \frac{n^2}{40}.
\] Using this estimate, we can proceed as in Lemma
\ref{lem_game_case1_lem2} to finish the proof. We omit the details.
\end{proof}

\begin{thm}
There exists a constant $c$ such that
the following holds for every $b \le \frac{cn}{\log n}$.
If $G$ is a Dirac graph satisfying $(iii)$ of Lemma \ref{lem:dirac},
then Maker has a winning strategy for the $(1:b)$-Maker-Breaker Hamiltonicity game.
\end{thm}
\begin{proof}
Let $\beta = \beta_{\ref{lem_game_case3_lem1}}$, and $c =
\min\{c_{\ref{lem_game_case3_lem1}},
c_{\ref{lem_game_case3_lem2}}(\beta)\}$. By Lemma
\ref{lem_game_case3_lem1}, Maker can construct a graph with property
$\BIPRE(\frac{1}{6})$ in the first $\beta n \log n$ rounds. Then by
Lemma \ref{lem_game_case3_lem2}, Maker can construct a graph which
complements it in the remaining rounds. Therefore by Proposition
\ref{prop:bip-hamiltonian}, Maker can construct a Hamilton cycle and
win the game.
\end{proof}

\section{Concluding Remarks}

As we mentioned in the introduction, several measures of robustness of graphs
with respect to various graph properties have already been considered before.
In this paper, we propose two new measures, and
strengthen Dirac's theorem according to these measures. Our first result
asserts that there exists a
constant $C$ such that for $p \ge \frac{C \log n}{n}$ and an
arbitrary Dirac graph $G$ on $n$ vertices, if one takes its edges independently at
random with probability $p$, then the resulting graph is
a.a.s.~Hamiltonian. Our second theorem
says that if one plays a $(1:b)$ Maker-Breaker game on a Dirac
graph, then the critical bias for Maker's win is of order of
magnitude $\frac{n}{\log n}$. We proved both of these theorems under
one general framework.

It is worth comparing our results with two previous robustness
results on Dirac graphs. Given a graph $G$, let $h(G)$ be the number
of Hamilton cycles in $G$. Cuckler and Kahn \cite{CuKa}, confirming
a conjecture of S\'ark\"ozy, Selkow, and Szemer\'edi \cite{SaSeSz},
proved that $h(G) \ge \frac{n!}{(2+o(1))^n}$ holds for every Dirac
graph $G$. Since the expected number of Hamilton cycles in the graph
$G_p$ is $h(G) p^n$, our first result which implies $h(G) p^n \ge 1$
for $p \ge \frac{C\log n}{n}$, recovers a slightly weaker
inequality $h(G) \ge \left(\frac{n}{C\log n}\right)^n$. Another
result of Lee and Sudakov \cite{LeSu} states that for $p \gg
\frac{\log n}{n}$, every subgraph of $G(n,p)$ of minimum degree at
least $\left(\frac{1}{2}+o(1)\right)np$ contains a Hamilton cycle.
Even though there is no direct implication between that result and
our result, they are nevertheless very closely related. Indeed,
the result in \cite{LeSu} is similar in spirit to a slightly weaker version of our
theorem, which says that for every fixed positive real $\varepsilon$
and every graph $G$ of minimum degree at least $\left(\frac{1}{2} +
\varepsilon\right)n$, the graph $G_p$ is a.a.s.~Hamiltonian, since
the resulting graph can be considered as a subgraph of $G(n,p)$ of
minimum degree at least
$\left(\frac{1}{2}+\frac{\varepsilon}{2}\right)np$.

We believe that two new measures of robustness that we proposed in this
paper, taking a random subgraph, and playing a Maker-Breaker game,
can be used to further study many other classical
Extremal Graph and Hypergraph Theory results as well.

\end{document}